\documentclass[a4paper,11pt]{article}

\usepackage[latin1]{inputenc}

\usepackage{latexsym}

\usepackage{amssymb}
\usepackage{amsfonts}
\usepackage{amsmath}
\usepackage{amsthm}
\usepackage{textcomp}

\newtheorem{theorem}{Theorem}[section]    
\newtheorem{lemma}[theorem]{Lemma}          
\newtheorem{corollary}[theorem]{Corollary}
\newtheorem{proposition}[theorem]{Proposition}

\theoremstyle{definition}

\title{A compactification for the spaces of convex projective structures on manifolds}
\author{Daniele Alessandrini \\ \small daniele.alessandrini@gmail.com}
\date{}

\newcommand{\nuovo}[1]{{{\bfseries \upshape #1}}}

\newcommand{\enne}{{\mathbb{N}}}
\newcommand{\ze}{{\mathbb{Z}}}
\newcommand{\qu}{{\mathbb{Q}}}
\newcommand{\erre}{{\mathbb{R}}}
\newcommand{\ci}{{\mathbb{C}}}

\newcommand{\pro}{{\mathbb{P}}}
\newcommand{\cappa}{{\mathbb{K}}}
\newcommand{\effe}{{\mathbb{F}}}

\newcommand{\iper}{{\mathbb{H}}}
\newcommand{\ident}{{\mbox{Id}}}
\newcommand{\tro}{{\mathbb{T}}}

\newcommand{\anello}{{\mathcal{A}}}
\newcommand{\ocors}{{\mathcal{O}}}

\newcommand{\rappr}{{\mathcal{R}}}

\newcommand{\famil}{{\mathcal{F}}}
\newcommand{\gfamil}{{\mathcal{G}}}

\newcommand{\ameba}{{\mathcal{A}}}

\DeclareMathOperator{\trace}{tr}

\DeclareMathOperator{\homom}{Hom}
\DeclareMathOperator{\charat}{Char}

\DeclareMathOperator{\card}{Card}
\DeclareMathOperator{\Log}{Log}

\DeclareMathOperator{\adjoint}{Ad}

\DeclareMathOperator{\spec}{Spec}
\DeclareMathOperator{\hol}{hol}
\DeclareMathOperator{\diff}{Diff}

\newenvironment{gmatrice}{ \begin{pmatrix} }{ \end{pmatrix} }
\newenvironment{pmatrice}{ \left( \begin{smallmatrix} }{ \end{smallmatrix}  \right) }

\newcommand{\freccia}{{\ \longrightarrow \ }}

\newcommand{\tende}{{\ \rightarrow \ }}

\newcommand{\numero}{{\textsection}}

\usepackage[dvips]{graphicx}

\newcommand{\figuresize}{4.9cm} 

\begin{document}

\sloppy

\maketitle



\section{Introduction}

Let $M$ be a closed oriented $n$-manifold such that $\pi_1(M)$ is virtually centerless, it is Gromov-hyperbolic and it is torsion-free. We denote by $\mathcal{T}^{c}_{\erre\pro^n}(M)$ the parameter space of marked convex projective structures on $M$. In this paper we construct a compactification $\mathcal{T}^{c}_{\erre\pro^n}(M) \cup \partial \mathcal{T}^{c}_{\erre\pro^n}(M)$ such that the action of the mapping class group of $M$ extends continuously to an action on the boundary. The construction generalizes the construction of compactification of Teichm\"uller spaces, and is made by using some techniques of tropical geometry we developed in \cite{A1}. An interpretation of the boundary points is then given in \cite{A2}.

If $S$ is an orientable hyperbolic surface of finite type, we denote by $\mathcal{T}^{cf}_{\iper^2}(S)$ the Teichm\"uller space of $S$. The original construction of compactification of $\mathcal{T}^{cf}_{\iper^2}(S)$ was made by Thurston, who constructed a boundary $\partial \mathcal{T}^{cf}_{\iper^2}(S)$ with a natural action of the mapping class group of $S$ preserving a natural piecewise linear structure on $\partial \mathcal{T}^{cf}_{\iper^2}(S)$.

Later, Morgan and Shalen constructed the same boundary using different techniques, mostly from algebra and complex algebraic geometry (see \cite{MS84}, \cite{MS88}, \cite{MS88'}). In their construction they identify the Teichm\"uller space with a connected component of the real part of a complex algebraic set, namely the variety of characters of representations of $\pi_1(S)$ in $SL_2(\ci)$. This variety was generated by the trace functions $I_\gamma$, $\gamma \in \pi_1(S)$, and the compactification was made by taking the limit points of the ratios of the values ${\left[\log(|I_\gamma(x)|+2)\right]}_{\gamma \in \pi_1(S)}$. With every boundary points they associated a valuation of the field of fractions of an irreducible component of the character variety, and this valuation defined in a natural way an action of $\pi_1(S)$ on a real tree (a generalization of an ordinary simplicial tree), whose projectivized spectrum corresponded to the boundary point. They also showed that every action of $\pi_1(S)$ on a real tree induced, dually, a measured lamination on $S$, recovering the interpretation of Thurston.

Note that the Teichm\"uller space has the structure of connected component of a real algebraic set, while the boundary has a piecewise linear structure. Moreover points of the interior part correspond to Riemann surfaces (complex algebraic curves), while the boundary points correspond to actions on real trees, a generalization of simplicial trees, that are polyhedral objects. Hence both the construction of the boundaries and the interpretation of boundary points seem to be the effect of a degeneration from an algebraic object to a polyhedral object. This recalls directly some features of tropical geometry, the geometry of the tropical semifield, a semifield $\tro$ with $\erre$ as the underlying set and $\max$ and $+$ as, respectively, addiction and multiplication. There is a deformation, the Maslov dequantization, of the semifield $\erre_{> 0}$ in the tropical semifield, made by taking logarithms of real numbers with increasing bases. This deformation transform algebraic varieties into tropical varieties, that are polyhedral subsets of $\erre^n$. A tropical variety can be described as the image, through the componentwise valuation map, of an algebraic variety over a non-archimedean field. We worked on this subject in \cite{A1}.

Here we present a construction of compactification, born in the framework of tropical geometry, that generalizes the construction of \cite{MS84}. We work directly with real semi-algebraic sets, and this allows us to consider more general families of functions. Given a semi-algebraic set $V$ and a proper family of positive continuous semi-algebraic functions ${\{f_i\}}_{i \in I}$, we can construct a compactification of $V$ by applying the Maslov dequantization to it, more precisely we use an inverse system of logarithmic limit sets, looking at the asymptotic behavior of the functions $\log(f_i)$ (see \cite{A1} and section \ref{sez:compactification}).

The properties of the Morgan-Shalen compactification can be extended to this more general setting: if the family ${\{f_i\}}_{i \in I}$ is invariant for the action of a group, the action can be extended continuously to the boundary, and the boundary is the image of an extension of $V$ to a non-archimedean field. Moreover, using the fact that logarithmic limit sets are polyhedral complexes (see \cite{A1}), under some hypotheses it is possible to give a natural piecewise linear structure to the boundary. See section \ref{sez:properties} for details.

To apply this construction to the space $\mathcal{T}^{c}_{\erre\pro^n}(M)$, we need to put a structure of semi-algebraic set on it. To do this we need study the set of characters of representations of a finitely generated group $\Gamma$ in $SL_{n+1}(\erre)$, denoted by $\overline{\charat}(\Gamma, SL_{n+1}(\erre))$. We show that this set can be identified with a closed semi-algebraic subset of an affine space $\erre^m$, and we prove some other properties in theorem \ref{teo:real char} and corollary \ref{corol:real char}. See section \ref{sez:realchar} for details.

Then we recall the definition of the parameter spaces for marked and based projective structures on an $n$-manifold $M$ as above. We study the relations between these spaces and the spaces $\overline{\charat}(\Gamma, SL_{n+1}(\erre))$. More specifically, we prove that $\mathcal{T}^{c}_{\erre\pro^n}(M)$ is a closed semi-algebraic subsets of $\overline{\charat}(\pi_1(M), SL_{n+1}(\erre))$. This fact is based on some deep results of Benoist (\cite{Be1}, \cite{Be2}, \cite{Be3}) and on corollary \ref{corol:real char}. Then we choose a family of semi-algebraic functions on $\mathcal{T}^{c}_{\erre\pro^n}(M)$ representing the translation lengths of the elements of $\pi_1(M)$ with respect to the Hilbert metric, and we use this family to construct the compactification. This family is invariant for the action of the mapping class group of $M$, hence the action of the mapping class group extends continuously to the boundary. See section \ref{sez:projective} for details.

In the companion paper \cite{A2} we investigate which objects can be used for the interpretation of the boundary points. Points of the interior part of $\overline{\charat}(\pi_1(M),SL_{n+1}(\erre))$ correspond to representations of $\pi_1(M)$ on $SL_{n+1}(\erre)$, or, geometrically, to actions of the group on a projective space of dimension $n$. Points of the boundary correspond instead to representations of the group in $SL_{n+1}(\cappa)$, where $\cappa$ is a non-archimedean field. In \cite{A2} we find a geometric interpretation of these representations, as actions of the group on tropical projective spaces of dimension $n$.

Finally, in section \ref{sez:Teichmuller}, we apply our construction to the Teichm\"uller spaces, as in \cite{MS84}. The boundary constructed in this way is the Thurston boundary, endowed with a natural piecewise linear structure, that is equivalent to the one defined by Thurston. This shows how the piecewise linear structure on the boundary is induced by the semi-algebraic structure on the interior part.

\section{Compactification of semi-algebraic sets}        \label{sez:compactification}

In this section, we apply the Maslov dequantization to a semi-algebraic set, constructing a compactification. More precisely, we construct an inverse systems of logarithmic limit sets, whose inverse limit we use to construct the compactification. This inverse limit is a cone, and it is the image of the extension of the space to a real closed non-archimedean field under a componentwise valuation map.

This construction is a generalization of the Morgan-Shalen compactification (see \cite{MS84}). They worked in the framework of complex algebraic geometry: given a complex variety $V$ and countable family of polynomial functions ${\{f_i\}}_{i \in I}$ containing a set of generators of the ring of coordinates of $V$, they constructed a compactification of $V$ taking the limit points of the ratios of the values ${\left[\log(|f_i(x)|+2)\right]}_{i \in I}$. Here we work directly with real semi-algebraic sets, and this allows us to consider more general families of functions. Given a semi-algebraic set $V$ and a proper family of positive continuous semi-algebraic functions ${\{f_i\}}_{i \in I}$, we can construct a compactification of $V$, in a way that is shown to be equivalent to taking the limit points of the ratios of the values ${\left[\log(f_i(x))\right]}_{i \in I}$. The properties of the Morgan-Shalen compactification can be extended to this more general setting. This generalization is important for the construction of the compactification for the spaces of convex projective structures on a closed $n$-manifold, as the family of functions we consider there is not a family of polynomial functions.

A compactification construction related to this one is the one in \cite{Te}, called tropical compactification. The two compactifications are different, for example the tropical compactification is a complex variety and the boundary is a divisor, while the compactifications we present here don't have a structure of algebraic variety, here the objects we put on the boundary have a polyhedral nature. Anyway the two notions are related, as we show in subsection \ref{subsez:injective}.

\subsection{Embedding-dependent compactification}

Let $V \subset {(\erre_{>0})}^n$ be a closed semi-algebraic set. We can construct a compactification for $V$, using its logarithmic limit set $\ameba_0(V) \subset \erre^n$, a polyhedral fan.

This fan represents the behavior at infinity of the amoeba, hence it can be used to compactify it. We take the quotient by the spherical equivalence relation
$$x \sim y \Leftrightarrow \exists \lambda >0: x = \lambda y$$
and we get the boundary 
$$\partial V = (\ameba_{0}(V)\setminus\{0\})/\sim \ \ \subset S^{n-1}$$

Now we glue $\partial V$ to $V$ at infinity in the following way. We compactify $\erre^n$ by adding the sphere at infinity:
$$\erre^n \ni x \freccia \frac{x}{\sqrt{1+{\|x\|}^2}} \in D^n\hspace{2cm}D^n \approx \erre^n \cup S^{n-1}$$

Given a $t_0 < 1$, we will denote by $\overline{V}$ the closure of $\ameba_{t_0}(V)$ in $D^n$. Then 
$$\overline{V} = \ameba_{t_0}(V) \cup \partial V$$

\begin{proposition}
The map $\Log_{\left(\frac{1}{t_0}\right)} : V \freccia \overline{V}$ is a compactification of $V$. The compactification does not depend on the choice of $t_0$.
\end{proposition}

\begin{proof}
The map $\Log_{\left(\frac{1}{t_0}\right)}$ is a homeomorphism between $V$ and $\ameba_{t_0}(V)$. The set $\ameba_{t_0}(V)$ is closed in $\erre^n$, hence its closure is the union of $\ameba_{t_0}(V)$ and a subset of $S^{n-1}$. As $S^{n-1}$ is closed in $D^n$, $\ameba_{t_0}(V)$ is open and dense in $\overline{V}$. 
\end{proof}

Note that the logarithmic limit set $\ameba_0(V)$ is the cone over the boundary, and for this reason it will sometimes be denoted by $C(\partial V)$.

\subsection{Embedding-independent compactification}

This construction can be generalized in a way that does not depend on the immersion of $V$ in $\erre^n$. Let $V \subset \erre^n$ be a semi-algebraic set. A finite family of continuous semi-algebraic functions $\famil = \{f_1, \dots, f_m\}$, with $f_i:V \freccia \erre_{>0}$, is called a \nuovo{proper family} if the map
$$E_\famil : V \ni x \freccia (f_1(x), \dots, f_m(x)) \in {(\erre_{>0})}^m$$
is proper. In this case the map $L_\famil = \Log_{\left(\frac{1}{t_0}\right)} \circ E_\famil$ is also proper.

The image $E_\famil(V) \subset {(\erre_{>0})}^n$ is a closed semi-algebraic subset, and we can compactify it as before, by $\overline{E_\famil(V)} = \ameba_{t_0}(E_\famil(V)) \cup \partial E_\famil(V)$.

Let $\hat{V} = V \cup \{\infty\}$ denote the Alexandrov compactification of $V$. Consider the map
$$i:V \ni x \freccia (x, L_\famil(x)) \in \hat{V} \times \overline{E_\famil(V)}$$
and let $\overline{V}_\famil$ be the closure of the image $i(V)$ in $\hat{V} \times \overline{E_\famil(V)}$.

\begin{proposition} 
The map $i:V \freccia \overline{V}_\famil$ is a compactification of $V$. The boundary $\partial_\famil V = \overline{V}_\famil \setminus i(V)$ is the set $\partial E_\famil(V)$.
\end{proposition}

\begin{proof}
The image of $i$ is homeomorphic to $V$ (the inverse being the projection $p_1$ on the first factor). The space $\overline{V}_\famil$ is compact as it is closed in a compact. The complement of $i(V)$ in $\overline{V}_\famil$ is the set $p_1^{-1}(\infty) \cap \overline{V}_\famil$, a closed set. Hence the map $i:V \freccia \overline{V}_\famil$ is a topological immersion of $V$ in a open dense subset of a compact space, i.e. a compactification.

The boundary is $\partial_\famil V = p_1^{-1}(\infty) \cap \overline{V}_\famil$. The projection $p_2$ on the second factor identifies $\partial_\famil V$ with a subset of $\overline{E_\famil(V)}$. As the map $E_\famil$ is proper, we have $\partial_\famil V = \partial E_\famil(V)$. 
\end{proof}

The cone over the boundary will be denoted by $C(\partial_\famil V) = \ameba_0( E_\famil(V) )$.

\subsection{Limit compactification}

Here we present a further generalization of the construction of the compactification. This generalization is needed if we want to extend the action of a group on the semi-algebraic set to an action on the compactification, as in subsection \ref{subsez:action extension}.

Let $V \subset \erre^n$ be a semi-algebraic set. A (possibly infinite) family of continuous semi-algebraic functions $\gfamil = {\{f_i\}}_{i \in I}$, with $f_i:V \freccia \erre_{>0}$, is called a \nuovo{proper family} if there exist a finite subfamily $\famil \subset \gfamil$ that is proper.

Suppose that $\gfamil$ is proper. Let
$$P_\gfamil = \{ \famil \subset \gfamil \ |\ \famil \mbox{ is proper } \}$$
a non-empty set partially ordered by inclusion. If $\famil \subset \famil'$ we denote by $\pi_{\famil',\famil}$ the projection
$$\pi_{\famil',\famil}:\erre^{\famil'} \freccia \erre^{\famil}$$
on the coordinates corresponding to $\famil$. This projection restricts to a surjective map
$$\pi_{\famil',\famil | \ameba_{t_0}( E_{\famil'}(V) )}:\ameba_{t_0}( E_{\famil'}(V) ) \freccia \ameba_{t_0}( E_\famil(V) )$$
By \cite[prop. 4.7]{A1}, the restriction to the logarithmic limit sets is also surjective:
$$\pi_{\famil',\famil | \ameba_{0}( E_{\famil'}(V) )}:\ameba_{0}( E_{\famil'}(V) ) \freccia \ameba_{0}( E_\famil(V) )$$

\begin{proposition}
Let $\famil, \famil' \in P_\gfamil$. If $\famil \subset \famil'$, the map $\pi_{\famil',\famil | \ameba_{0}( E_{\famil'}(V))}$ induces a map 
$$\partial \pi_{\famil',\famil}: \partial_{\famil'} V \freccia \partial_\famil V$$
\end{proposition}

\begin{proof}
We have to prove that
$${\left(\pi_{\famil',\famil | \ameba_{0}( E_{\famil'}(V) )}\right)}^{-1}(0) = \{0\}$$
Let $x \in \ameba_{0}( E_{\famil'}(V) ) \setminus \{0\}$, hence the set $\widetilde{\famil'} = \{ f \in \famil' \ |\ x_f \neq 0\}$ is not empty. The image of $x$ in the quotient $\partial_{\famil'}(V)$ is the point $[x]$. As $\overline{V}_{\famil'}$ is a compactification, there exists a sequence ${(x_n)} \subset V$ such that $(x_n) \tende [x]$ in $\overline{V}_{\famil'}$. As in \cite[sec. 3]{A1} we can choose $(x_n)$ such that, for all function $f \in \famil' \setminus \widetilde{\famil'}$, the sequence $\log_{\left(\frac{1}{t_0}\right)}(f(x_n))$ is bounded. As the map $L_\famil$ is proper, the sequence $(L_\famil(x_n))$ is not bounded, hence there is a function $f \in \famil$ such that $\log_{\left(\frac{1}{t_0}\right)}(f(x_n))$ is unbounded. Hence the corresponding coordinate $x_f \neq 0$, and $\pi_{\famil',\famil}(x) \neq 0$.
\end{proof}

The maps $\pi_{\famil',\famil}$ and $\partial \pi_{\famil',\famil}$ define inverse systems ${\{\ameba_{t_0}( E_\famil(V) )\}}_{\famil \in P_\gfamil}$, ${\{\ameba_{0}( E_\famil(V) )\}}_{\famil \in P_\gfamil}$ and ${\{\partial_\famil V \}}_{\famil \in P_\gfamil}$. Consider the inverse limit
$$ L = \lim_{\longleftarrow} \ameba_{t_0}( E_\famil(V) )$$
we will denote by $\pi_{\gfamil,\famil} : L \freccia \ameba_{t_0}( E_\famil(V) )$ the canonical projection. By the explicit description of the inverse limit, $L$ is a closed subset of the product:
$$\left\{ {(x_\famil)} \in \prod_{\famil \in P_\gfamil} \ameba_{t_0}( E_\famil(V) ) \ |\ \forall \famil \subset \famil' : \pi_{\famil',\famil}(x_{\famil'}) = x_\famil \right\} $$

For every $x \in L$, and every $f \in \gfamil$, let $\famil$ be a proper finite family containing $f$. Then the value of the $f$-coordinate of the point $\pi_{\gfamil,\famil}(x)$ does not depend on the choice of the family $\famil$. This value will be denoted by $x_f$. The map
$$L \ni x \freccia {(x_f)}_{f \in \gfamil} \in \erre^\gfamil$$
identifies $L$ with a subset of $\erre^{\gfamil}$.

The system of maps $L_\famil : V \freccia \ameba_{t_0}( E_\famil(V) )$, defined for every $\famil \in P_\gfamil$, induces by the universal property a well defined map $L_{\gfamil}: V \freccia L$. 

\begin{proposition}
The map $L_{\gfamil}$ is surjective and proper, and it can be identified with the map  
$$V \ni x \freccia {\left(\log_{\left(\frac{1}{t_0}\right)}(f(x))\right)}_{f \in \gfamil} \in \erre^{\gfamil} $$
\end{proposition}

\begin{proof}
We proceed by transfinite induction on the cardinality of $\gfamil$. If $\gfamil$ is finite, the statement is trivial, as in this case $\gfamil$ is the maximum of $P_\gfamil$, the inverse limit is simply $L = \ameba_{t_0}(E_\gfamil(V))$ and we now that for $\famil \in P_\gfamil$ the maps $L_\famil : V \freccia \ameba_{t_0}( E_\famil(V) )$ are surjective and proper.

Now suppose, by induction, that the statement is true for all proper families with cardinality less than the cardinality of $\gfamil$. We denote by $P_\gfamil'$ the set of all proper subfamilies of $\gfamil$ with smaller cardinality. Let $y \in L$. By the inductive hypothesis for every $\famil \in P_\gfamil'$, the map $L_\famil : V \freccia \ameba_{t_0}( E_\famil(V) )$ is surjective and proper, hence the inverse image $L_\famil^{-1}(\pi_{\gfamil,\famil}(y))$ is a compact and non-empty subset of $V$.

By the Zermelo theorem every set has a cardinal well ordering, i.e. a well ordering such that all initial segments have cardinality less than the cardinality of the set. Moreover we can choose a finite set that will be an initial segment for the well ordering. We choose a cardinal well ordering $\prec$ of $\gfamil$ such that a finite proper family $\overline{\famil}$ is an initial segment. Consider the set $Q_\gfamil$ of all initial segments of $\gfamil$ containing $\overline{\famil}$.

We have $Q_\gfamil \subset P_\gfamil'$, hence for every $\famil \in Q_\gfamil$ the subset $L_\famil^{-1}(\pi_{\gfamil,\famil}(y)) \subset V$ is compact and non-empty, and if $\famil \subset \famil'$, then $L_\famil^{-1}(\pi_{\gfamil,\famil}(y)) \subset L_{\famil'}^{-1}(\pi_{\gfamil,{\famil'}}(y))$. The sets $L_\famil^{-1}(\pi_{\gfamil,\famil}(y))$ are nested compact subsets of $V$, hence their intersection is non-empty:
$$\bigcap_{\famil \in Q_\gfamil} L_\famil^{-1}(\pi_{\gfamil,\famil}(y)) \neq \emptyset$$
If $x$ is a point in this intersection, then $L_\gfamil(x) = y$. The fact that the map $L_\gfamil$ can be identified with the map ${\left(\log_{\left(\frac{1}{t_0}\right)}(f(x))\right)}_{f \in \gfamil}$ is clear, and this implies that the map is proper.
\end{proof}

As the map $L_\gfamil$ is surjective, in the following we will denote $L$ by $L_{\gfamil}(V)$. Now consider the inverse limit
$$M =  \lim_{\longleftarrow} \overline{E_\famil(V)} = \lim_{\longleftarrow} \ameba_{t_0}( E_\famil(V) ) \cup \partial_\famil V$$
The space $M$ is compact, as it is an inverse limit of compact spaces, and we will use the map $L_\gfamil: V \freccia M$ to define a compactification, as in the previous subsection.

Consider the map
$$i:V \ni x \freccia (x, L_{\gfamil}(x)) \in \hat{V} \times M$$
Let $\overline{V}_{\gfamil}$ be the closure of the image $i(V)$ in $\hat{V} \times M$.

\begin{proposition} 
The map $i:V \freccia \overline{V}_{\gfamil}$ is a compactification of $V$. The boundary $\partial_{\gfamil} V = \overline{V}_\gfamil \setminus i(V)$ is the set $\displaystyle \lim_{\longleftarrow} \partial_\famil V$.
\end{proposition}

\begin{proof}
As before, $i(V)$ is homeomorphic to $V$. The space $\overline{V}_{\gfamil}$ is compact as it is closed in a compact. The complement of $i(V)$ in $\overline{V}_{\gfamil}$ is the set $p_1^{-1}(\infty) \cap \overline{V}_{\gfamil}$, a closed set. Hence the map $i:V \freccia \overline{V}_{\gfamil}$ is a compactification.

Every map $L_\famil$ is proper, hence the boundary is precisely the set $\lim_{\longleftarrow} \partial_\famil V$.
\end{proof}

The limit $\partial_{\gfamil} V$ is the spherical quotient of the limit
$$C(\partial_\gfamil V) = \lim_{\longleftarrow} C(\partial_\famil V) = \lim_{\longleftarrow} \ameba_0( E_\famil(V) ) $$
More explicitly, $C(\partial_\gfamil V)$ is a closed subset of the product:
$$\left\{ {(x_\famil)} \in \prod_{\famil \in P_\gfamil} \ameba_{0}( E_\famil(V) ) \ |\ \forall \famil \subset \famil' : \pi_{\famil',\famil}(x_{\famil'}) = x_\famil \right\} $$

As before, for every $x \in C(\partial_\gfamil V)$, and every $f \in \gfamil$, let $\famil$ be a proper finite family containing $f$. Then the value of the $f$-coordinate of the point $\pi_{\gfamil,\famil}(x)$ does not depend on the choice of the family $\famil$. This value will be denoted by $x_f$. The map
$$C(\partial_\gfamil V) \ni x \freccia {(x_f)}_{f \in \gfamil} \in \erre^\gfamil$$
identifies $C(\partial_\gfamil V)$ with a closed subset of $\erre^{\gfamil}$.

\section{Properties of the boundary}        \label{sez:properties}

\subsection{Group actions and compactifications}   \label{subsez:action extension}

Let $G$ be a group acting with continuous semi-algebraic maps on a semi-algebraic set $V \subset \erre^n$. Suppose that $\gfamil$ is a (possibly infinite) proper family of functions $V \freccia \erre_{>0}$, and that $\gfamil$ is invariant for the action of $G$.

Then the action of $G$ on $V$ extends continuously to an action on the compactification $\overline{V}_{\gfamil}$.

As $\gfamil$ is invariant for the action of $G$, if we see the limits $L_{\gfamil}(V)$ and $C(\partial_\gfamil V)$ as subsets of $\erre^{\gfamil}$, then $G$ acts on $L_{\gfamil}(V)$ and $C(\partial_\gfamil V)$ by a permutation of the coordinates corresponding to the action on $\gfamil$, and this action induces an action on the spherical quotient of $C(\partial_\gfamil V)$, the boundary $\partial_{\gfamil}$.

Note that the map $L_{\gfamil}:V \freccia L_{\gfamil}(V)$ is equivariant for this action, hence the action of $G$ on $\partial_{\gfamil}$ extends continuously the action of $G$ on $V$.

\subsection{Injective families and the piecewise linear structure}  \label{subsez:injective}

Let $V \subset \erre^n$ be a semi-algebraic set, and let $\gfamil$ be a (possibly infinite) proper family of continuous semi-algebraic functions $V \freccia \erre_{>0}$.

An \nuovo{injective family} is a finite proper subfamily $\famil \subset \gfamil$ such that the canonical surjective map $\partial \pi_\famil: C(\partial_\gfamil V) \freccia C(\partial_{\famil} V)$ is also injective. 

The existence of an injective family is an important tool for us. For $\famil \in P_\gfamil$, the sets $\ameba_{0}( E_\famil(V) ) = C(\partial_\famil V)$ are all polyhedral complexes, by \cite[thm. 3.11]{A1}, and the maps
$$\pi_{\famil',\famil | C(\partial_{\famil'} V)}:C(\partial_{\famil'} V) \freccia C(\partial_\famil V)$$
are surjective piecewise linear maps.

\begin{proposition}
Suppose that $\gfamil$ contains an injective family. Consider the subset $Q_\gfamil \subset P_\gfamil$ of all injective families in $\gfamil$. Each of the maps $C(\partial_\gfamil V) \freccia C(\partial_{\famil} V) \subset \erre^\famil$, for $\famil \in Q_\gfamil$, is a chart for a piecewise linear structure on $C(\partial_\gfamil V)$, and all these charts are compatible, hence they define a canonical piecewise linear structure on $C(\partial_\gfamil V)$. As $C(\partial_\gfamil V)$ is the cone on $\partial_{\gfamil} V$, a piecewise linear structure is also defined on $\partial_{\gfamil} V$.
\end{proposition}

\begin{proof}
If $\famil, \famil' \in Q_\gfamil$, then also $\famil \cup \famil' \in Q_\gfamil$. Then the natural projections $C(\partial_{\famil' \cup \famil} V) \freccia C(\partial_{\famil} V)$ and $C(\partial_{\famil' \cup \famil} V) \freccia C(\partial_{\famil'} V)$ are piecewise linear isomorphisms. Hence the charts are compatible.
\end{proof}

For example it is possible to construct a canonical compactification of a complex very affine variety. This construction is closely related to the tropical compactification of \cite{Te}. A \nuovo{very affine variety} $V \subset {(\ci^*)}^n$ is a closed algebraic subset of the complex torus. The identification $\ci = \erre^2$ turns $V$ into a real semi-algebraic subset of $\erre^{2n}$. We denote by $\ci[V]$ the ring of coordinates of $V$, and by $\ci[V]^*$ the group of invertible elements, i.e. the set of polynomials that never vanish on $V$. For every $f \in \ci[V]^*$ we denote by $|f|$ the continuous semi-algebraic function:
$$|f|:V \ni x \freccia |f(x)| \in \erre_{>0} $$
We choose a proper family $\gfamil$ in the following way:
$$\gfamil = \{ |f| \ |\ f \in \ci[V]^* \} $$
As $V$ is an algebraic subset of the torus, the ring of coordinates is generated by invertible elements, for example the coordinate functions $X_1, \dots, X_n$ and their inverses $X_1^{-1}, \dots, X_n^{-1}$. Then the finite family $\{|X_1|, \dots, |X_n|\} \subset \gfamil$ is a proper family, hence also $\gfamil$ is a proper family. This family defines a compactification
$$i:V \freccia \overline{V}_{\gfamil} $$

Let $G$ be the group of all complex polynomial automorphisms of $V$. In particular $G$ acts on the semi-algebraic set $V$ with continuous semi-algebraic maps. This action preserves $\ci[V]^*$, hence it also preserves the family $\gfamil$. Then the action of $G$ on $V$ extends to an action on the compactification $\overline{V}_{\gfamil}$.

The family $\gfamil$ is infinite (and uncountable), yet it is possible to find an injective family. To do this we use the same technique used in \cite{Te} to construct the intrinsic torus. By \cite[Rem. 2.10]{ST}, the group $\ci[V]^* / \ci^*$ is finitely generated.

\begin{proposition}
Let $f_1, \dots, f_m \in \ci[V]^*$ be representatives of generators of the group $\ci[V]^* / \ci^*$. Then the family $\famil = \{|f_1|, \dots, |f_m|\} \subset \gfamil$ is an injective family.
\end{proposition}

\begin{proof}
Consider the projection $\pi_{\gfamil,\famil | C(\partial_{\gfamil} V)}:C(\partial_{\gfamil} V) \freccia C(\partial_{\famil} V)$ and suppose, by contradiction, that it is not injective. Then there exists an $x \in C(\partial_{\famil} V)$ such that $\pi_{\gfamil,\famil | C(\partial_{\gfamil} V)}^{-1}(x)$ contains at least two elements $y,y' \in C(\partial_{\gfamil} V)$. As $y \neq y'$ there exists an element $f \in \ci[V]^*$ such that the coordinates $y_{|f|}$ and $y_{|f|}'$ differ.

This means that also the projection $$\pi_{\famil \cup \{|f|\},\famil | C(\partial_{\famil\cup\{|f|\}} V)}:C(\partial_{\famil\cup\{|f|\}} V) \freccia C(\partial_{\famil} V)$$ is not injective. The set $C(\partial_{\famil\cup\{|f|\}} V)$ is simply the logarithmic limit set $\ameba_0(E_{\famil\cup\{|f|\}}(V)) \subset \erre^{m+1}$, and, by \cite[cor. 4.5]{A1} it is the image, under the componentwise valuation map, of the set $\overline{E_{\famil\cup\{|f|\}}(V)} \subset H({\overline{\erre}}^{\erre})$.

By the hypothesis on $f_1, \dots, f_m$, there exists integers $e_i \in \ze$ and a number $c \in \ci$ such that $f = c \prod_{i = 1}^m f_i^{e_i}$. Hence $|f| = |c| \prod_{i = 1}^m |f_i|^{e_i}$, and for every $z \in H({\overline{\erre}}^{\erre})$, $v(|f(z)|) = \sum_{i = 1}^m e_i v(|f_i(z)|)$. The valuation of $|f(z)|$ is determined by the valuations of $|f_i(z)|$, hence the map $\pi_{\famil \cup \{|f|\},\famil | C(\partial_{\famil\cup\{|f|\}} V)}$ is injective, a contradiction.
\end{proof}

\begin{corollary}
In particular the boundary $\partial_{\gfamil} V$ of a very affine variety has a natural piecewise linear structure, and the group $G$ acts on the compactification $\overline{V}_{\gfamil}$ with an action by complex polynomial maps on the interior part, and by piecewise linear maps on the boundary.
\end{corollary}

Unfortunately, when working with real algebraic sets, a general technique of this kind for constructing injective families does not work. If $V$ is a real algebraic set, a polynomial function that never vanish on $V$ is not necessarily invertible in the ring of coordinates. It is invertible in the ring of regular functions, but this ring is not always a finitely generated $\erre$-algebra, hence the group of invertible elements is not finitely generated.

The following proposition gives a sufficient hypothesis for the existence of injective families.

\begin{proposition}   \label{prop:positive polynomials}
Let $V \subset \erre^n$ be a real semi-algebraic set, and let $\gfamil$ be a proper family of positive continuous semi-algebraic functions on $V$. Suppose that there exists a proper family $\famil = \{f_1, \dots, f_m\} \subset \gfamil$ such that for every element $f \in \gfamil$ there exists a Laurent polynomial $P(x_1, \dots, x_m)$ with real and positive coefficients such that $f = P(f_1, \dots, f_m)$. Then $\famil$ is an injective family.
\end{proposition}

\begin{proof}
Consider the projection $\pi_{\gfamil,\famil | C(\partial_{\gfamil} V)}:C(\partial_{\gfamil} V) \freccia C(\partial_{\famil} V)$, and suppose, by contradiction, that it is not injective. Then there exists an $x \in C(\partial_{\famil} V)$ such that $\pi_{\gfamil,\famil | C(\partial_{\gfamil} V)}^{-1}(x)$ contains at least two elements $y,y' \in C(\partial_{\gfamil} V)$. As $y \neq y'$ there exists an element $f \in \gfamil$ such that the coordinates $y_{f}$ and $y_{f}'$ differ.

This means that also the projection $$\pi_{\famil \cup \{f\},\famil | C(\partial_{\famil\cup\{f\}} V)}:C(\partial_{\famil\cup\{f\}} V) \freccia C(\partial_{\famil} V)$$ is not injective. The set $C(\partial_{\famil\cup\{f\}} V)$ is simply the logarithmic limit set $\ameba_0(E_{\famil\cup\{f\}}(V)) \subset \erre^{m+1}$, and, by \cite[cor. 4.5]{A1} it is the image, under the componentwise valuation map, of the set $\overline{E_{\famil\cup\{f\}}(V)} \subset H({\overline{\erre}}^{\erre})$.

By hypothesis we have $f = P(f_1, \dots, f_m)$, where the coefficients of $P$ are real and positive, hence for every $z \in H({\overline{\erre}}^{\erre})$, the valuation of $f(z)$ is determined by the valuations of $f_i(z)$, hence the map $\pi_{\famil \cup \{|f|\},\famil | C(\partial_{\famil\cup\{|f|\}} V)}$ is injective, a contradiction.
\end{proof}

Note that in the previous proposition we had to require the Laurent polynomial $P(x_1, \dots, x_n)$ to have positive coefficients. With the weaker hypothesis that the polynomial $P$ is positive whenever the variables $x_1, \dots, x_n$ are positive, the statement becomes false. For example, let $V = {(\erre_{>0})}^2$. The family of functions $\famil = \{x,y\}$ is proper, and $C(\partial_\famil V)$ is simply the plane $\erre^2$. Consider the family $\gfamil = \{x,y,x^2+(y-1)^2\}$. The set $E_\gfamil(V) \subset \erre^3$ and the logarithmic limit set $C(\partial_\gfamil V)$ are represented in figure \ref{fig:plane}. The map $\pi_{\gfamil,\famil}$ is not injective, hence $\famil$ is not an injective family of $\gfamil$.

\begin{figure}[htbp]
\begin{center}
\includegraphics[width=\figuresize]{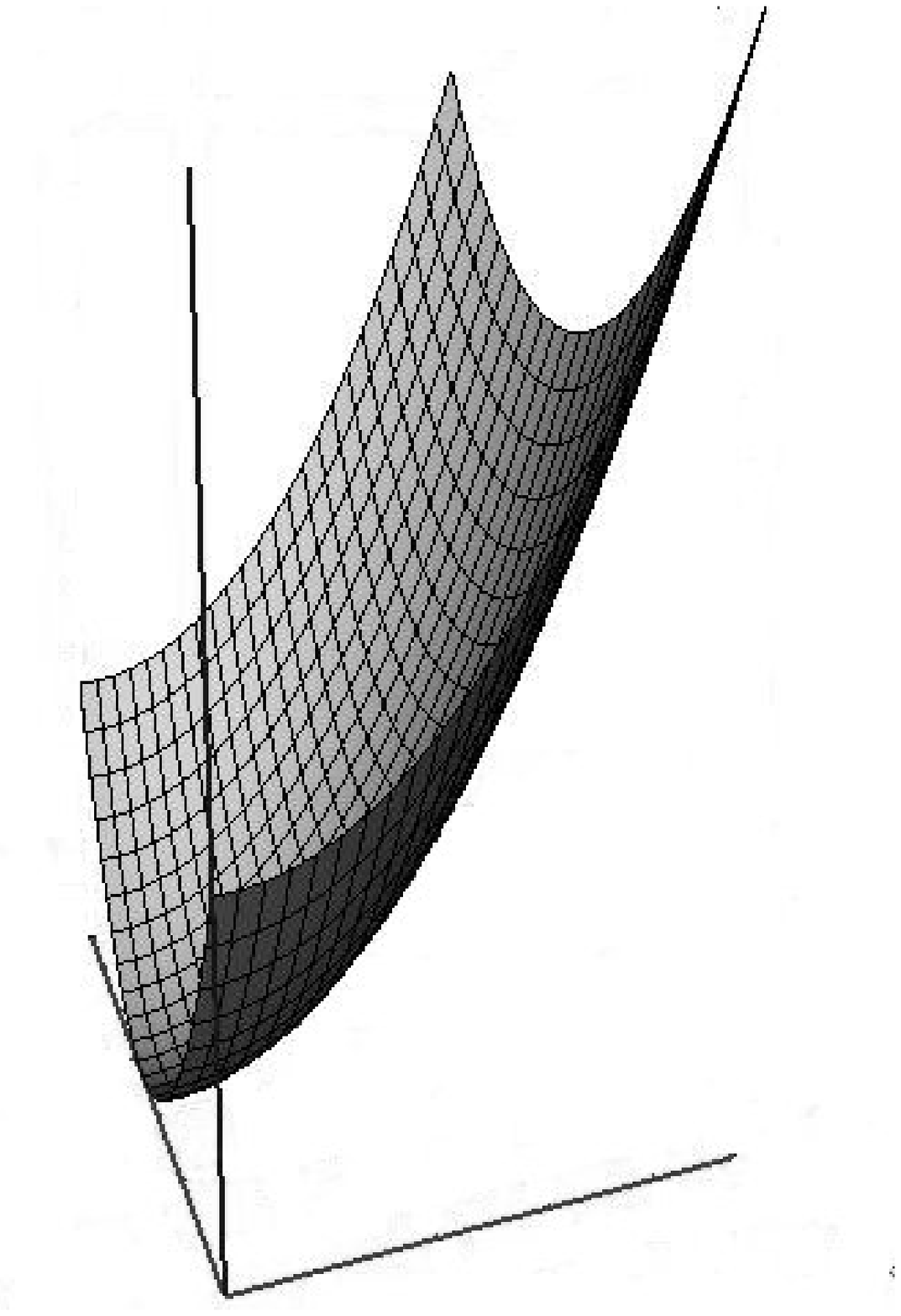}
\hspace{0.5cm}
\includegraphics[width=\figuresize]{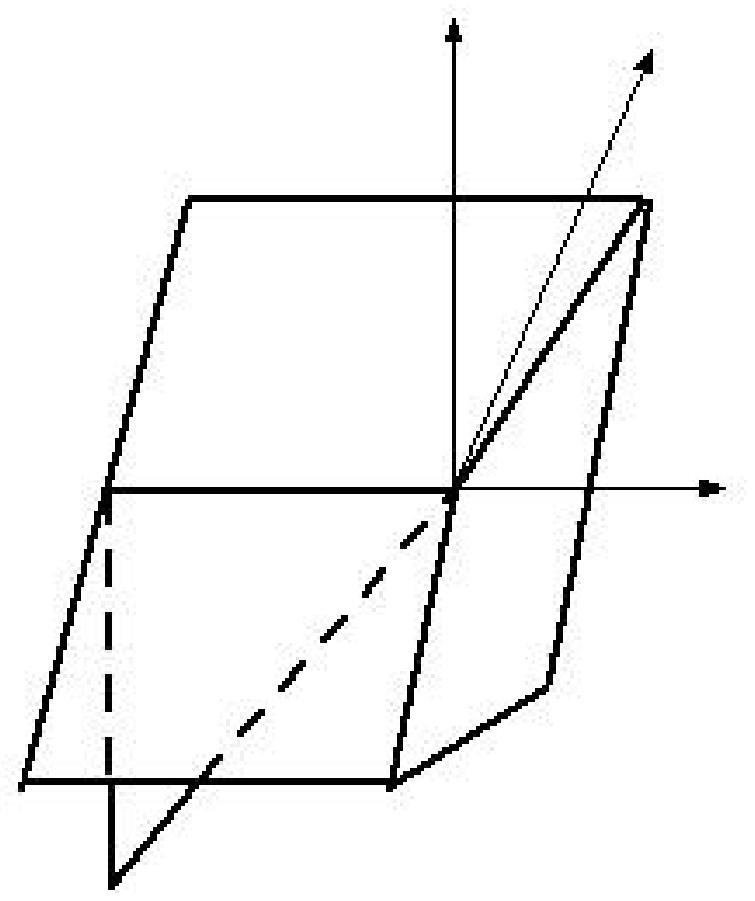}
\caption{$V = \{(x,y,z) \in {(\erre_{>0})}^3 \ |\  z = x^2 + (y-1)^2 \}$ (left picture), then $\ameba_0(V)$ is made up of four faces (right picture).}  \label{fig:plane}
\end{center}
\end{figure}

\subsection{Non-archimedean description}   \label{subsez:nonarchdescr}

Let $V \subset \erre^n$ be a semi-algebraic set, and let $\gfamil$ be a (possibly infinite) proper family of continuous semi-algebraic functions $V \freccia \erre_{>0}$.

Let $\effe$ be a real closed non-archimedean field with finite rank extending $\erre$. The convex hull of $\erre$ in $\effe$ is a valuation ring denoted by $\ocors_\leq$. This valuation ring defines a valuation $v:\effe^* \freccia \Lambda$, where $\Lambda$ is an ordered abelian group. As $\effe$ has finite rank, the group $\Lambda$ has only finitely many convex subgroups ${0} = \Lambda_0 \subset \Lambda_1 \subset \dots \subset \Lambda_r = \Lambda$. The number $r$ of convex subgroups is the \nuovo{rank} of the field $\effe$.

The quotient $\Lambda \freccia \Lambda / \Lambda_{r-1}$ is an ordered group of rank one, hence it is isomorphic to a subgroup of $\erre$. We fix one of these isomorphisms, and we denote by $\overline{v}$ the composition of the valuation $v$ with the quotient map $\Lambda \freccia \Lambda / \Lambda_{r-1}$, another valuation of $\effe$ that is real valued:
$$\overline{v}:\effe^* \freccia \erre$$

This valuation induces a norm and a log map:
$$\effe \ni h \freccia \parallel h\parallel=\exp(-\overline{v}(h)) \in \erre_{\geq 0}$$
$$\Log:{({\effe}_{>0})}^n \ni (h_1, \dots, h_n) \freccia \left(\log(\parallel h_1\parallel),\dots,\log(\parallel h_n\parallel)\right) \in \erre^n$$

Let $\overline{V}$ be the extension of $V$ to $\effe$, a semi-algebraic subset of ${(\effe_{>0})}^n$. Let $\overline{\gfamil} = \{ \overline{f} \ |\ f \in \gfamil\}$, where $\overline{f}:\overline{V} \freccia \effe_{>0}$ is the extension of the function $f:V \freccia \erre_{>0}$.

Let $\famil = \{f_1, \dots, f_m\} \subset \gfamil$ be a finite proper family, and let $\overline{\famil} = \{\overline{f_1}, \dots, \overline{f_m}\} \subset \overline{\gfamil}$ be the corresponding family of extensions. We will denote by $\overline{E_{\famil}}:\overline{V} \freccia {(\effe_{>0})}^m$ the extension of the map $E_\famil$.

\begin{proposition}
The image of the map
$$\Log:{(\effe_{>0})}^n \supset \overline{E_\famil}(\overline{V}) \ni x \freccia (-\overline{v}(x_1), \dots, -\overline{v}(x_n)) \in \erre^m$$
is contained in the logarithmic limit set $\ameba_0(E_\famil(V))$.
\end{proposition}

\begin{proof}
Let $t \in \effe$ be an element such that $t > 0$ and $\overline{v}(t) = 1$. Consider the subfield $\erre(t) \subset \effe$. The order induced by $\effe$ has the property that $t>0$ and $\forall x \in \erre_{>0}: t < x$. Hence $\effe$ contains the real closure of $\erre(t)$ with reference to this order, i.e. $H(\overline{\erre})$. Moreover the valuation $\overline{v}$ on $\effe$ restricts to the valuation we have defined on $H(\overline{\erre})$, as, if $\ocors_{\leq}$ is the valuation ring of $\effe$, $\ocors_{\leq} \cap H(\overline{\erre})$ is precisely the valuation ring $\ocors$ of $H(\overline{\erre})$.

We recall that, by corollary \cite[cor. 4.6]{A1}, $\Log( \overline{V} \cap H(\overline{\erre})) \subset \ameba_0(E_\famil(V))$. By reporting, word by word, the proof of \cite[thm. 4.3]{A1}, we can prove that the image $\Log(\overline{V})$ is contained in the closure of $\Log( \overline{V} \cap H(\overline{\erre}))$, i.e. it is contained in $\ameba_0(E_\famil(V))$.
\end{proof}

In other words, the image of the map
$$\Log_\famil = \Log \circ \overline{E_\famil}: \overline{V} \ni x \freccia (-\overline{v}(f_1(x)), \dots, -v(f_m(x))) \in \erre^m$$
is contained in $\ameba_0(E_\famil(V)) = C(\partial_\famil V)$.

The system of maps $\Log_\famil: \overline{V} \freccia C(\partial_\famil V)$, defined for every $\famil \in P_\gfamil$, induces by the universal property a well defined map $\Log_{\gfamil}: V \freccia C(\partial_\gfamil V)$. The map $\Log_{\gfamil}$ can be identified with the map
$$\overline{V} \ni x \freccia {\left( -\overline{v}(f(x))  \right)}_{f \in \gfamil} \in C(\partial_\gfamil V) \subset \erre^{\gfamil} $$

This map is not surjective for every field $\effe$. The aim of this subsection is to show that there exists a real closed non-archimedean field of finite rank $\effe$ such that the map $\Log_\gfamil$ is surjective. As a consequence, the boundary $\partial_\gfamil V$ is the spherical quotient of the set $\Log_\gfamil(\overline{V})$.

The easiest case is when $\gfamil$ has an injective family. In this case we use the Hardy field $H(\overline{\erre}^{\erre})$, as in \cite[subsec. 4.1]{A1}.

\begin{proposition}        \label{prop:nonarchinjective}
Suppose that the family $\gfamil$ contains an injective family $\famil$. Then if $\effe = H(\overline{\erre}^{\erre})$, the image of the map $\Log_\gfamil$ is the whole $C(\partial_\gfamil V)$.
\end{proposition}

\begin{proof}
Let $y \in C(\partial_\gfamil V)$. The map $\Log_\famil:\overline{V} \freccia \ameba_0(\overline{E_\famil}(V))$ is surjective by \cite[cor. 4.5]{A1}, hence there exists $x \in \overline{V}$ such that $\Log_\famil(x) = \pi_{\gfamil,\famil}(y)$. Hence $\pi_{\gfamil,\famil}(\Log_\gfamil(x)) = \pi_{\gfamil,\famil}(y)$, and, as $\famil$ is an injective family, $\Log_\gfamil(x) = y$.
\end{proof}

We can get a similar result even if we remove the hypothesis of the existence of an injective family. To do this we will need to use a larger field. If $V$ is a semi-algebraic set of dimension $r$, consider the group $\erre^r$, with the lexicographic order. We will use the field of transfinite Puiseaux series with real coefficients and exponents in $\erre^r$:
$$\erre((t^{\erre^r})) = \left\{ \sum_{r \in \erre^r} a_r t^r \ |\ a_r \in \erre, \{r \in \erre^r \ |\ a_r \neq 0\} \mbox{ is well ordered} \right\}$$
This is a real closed non-archimedean field of rank $r$, with a surjective valuation $v:\erre((t^\erre))^* \freccia \erre^r$, and a quotient surjective valuation $\overline{v}:\erre((t^\erre))^* \freccia \erre$.

The rest of this subsection is dedicated to the proof of theorem \ref{teo:non-arch descr}: if $\effe = \erre((t^{\erre^r}))$, the map $\Log_\gfamil$ is surjective.

Let $y \in C(\partial_\gfamil V)$, and let $[y] \in \partial_\gfamil V$ be the corresponding point. The set of all semi-algebraic subsets of $V$ is a Boolean algebra. Consider the set $\phi_y$ of all semi-algebraic subsets $A \subset V$ such that there exists a neighborhood $U$ of $[y]$ in $\overline{V}_\gfamil$ with $U \cap V \subset A$. The set $\phi_y$ is a \nuovo{filter} in the Boolean algebra of all semi-algebraic subsets of $V$, i.e. $\phi_y$ is closed for finite intersections, if $A \in \phi_y$ and $A \subset B$, then $B \in \phi_y$, $V \in \phi_y$, $\emptyset \not\in \phi_y$. The ultrafilter lemma states that every filter is contained in an \nuovo{ultrafilter}, i.e. a filter such that for every semi-algebraic set $A \subset V$, $A$ or $V \setminus A$ is in the ultrafilter. We choose an ultrafilter $\alpha$ containing $\phi_y$. Note that if $A \in \alpha$, then the complement of $A$ in $\overline{V}_\gfamil$ does not contain a neighborhood of $[y]$, hence the closure of $A$ in $\overline{V}_\gfamil$ contains the boundary point $[y]$. Intuitively speaking, the choice of an ultrafilter $\alpha$ containing $\phi_y$ can be interpreted as the choice of a way for converging to $[y]$.

Now, given an ultrafilter $\alpha$ as above, we construct a field $\mathfrak{K}(\alpha)$, as in \cite[subsec. 5.3]{Co}. Consider the ring $S(V)$ of all semi-algebraic functions from $V$ to $\erre$. Let $I(\alpha)$ be the subset of all functions whose zero locus is in $\alpha$. We define $\mathfrak{K}(\alpha) = S(V) / I(\alpha)$. It is easy to show that every non-zero element is invertible, hence $\mathfrak{K}(\alpha)$ is a field. Equivalently, $\mathfrak{K}(\alpha)$ can be defined as the quotient of $S(V)$ under the relation $f \sim g$ if and only if $\{f(x) = g(x)\} \in \alpha$. This second definition is the one used in \cite[subsec. 5.3]{Co}. We denote the equivalence class of a semi-algebraic function $f$ by $[f]$. In the same way we can define the order on $\mathfrak{K}(\alpha)$: $[f] \leq [g]$ if and only if $\{f(x) \leq g(x)\} \in \alpha$, and this definition does not depend on the choice of the representative. Hence $\mathfrak{K}(\alpha)$ is an ordered field, and has an $\mathcal{OS}$-structure.

\begin{proposition}
Given an $(L_\mathcal{OS})$-formula $\phi(x_1,\dots,x_n)$, and given definable functions $f_1,\dots,f_n$, we have:
$$\mathfrak{K}(\alpha) \vDash \phi([f_1], \dots, [f_n]) \Leftrightarrow \exists S \in \alpha: \forall t \in S: \overline{\erre} \vDash \phi(f_1(t),\dots,f_n(t))$$
\end{proposition}

\begin{proof}
See \cite[thm. 5.8]{Co}.
\end{proof}

In particular $\mathfrak{K}(\alpha)$ is a real closed field. For every element $a \in \erre$, the constant function with value $a$ defines an element of $\mathfrak{K}(\alpha)$ that is identified with $a$. This defines an an embedding $\erre \freccia \mathfrak{K}(\alpha)$.

Consider an embedding $V \freccia \erre^n$. For every semi-algebraic subset $A \subset V$, we denote by $\overline{A}$ its Zariski-closure in $\erre^n$. Consider the set
$$W_\alpha =  \bigcap_{A \in \alpha} \overline{A}$$
Every infinite intersection of Zariski-closed sets can be written as a finite intersection, hence $W_\alpha$ is a Zariski-closed subset of $\erre^n$. Moreover $W_\alpha$ is irreducible and $W_\alpha \cap V \in \alpha$. We will call the set $W_\alpha$ the \nuovo{algebraic support} of the ultrafilter $\alpha$. Note that the algebraic support depends on the chosen embedding $V \freccia \erre^n$.

Every polynomial $f \in \erre[x_1, \dots, x_n]$ defines a semi-algebraic function $f:V \freccia \erre$, and an element $[f] \in \mathfrak{K}(\alpha)$. Two polynomials define the same element of $\mathfrak{K}(\alpha)$ if and only if the coincide on $W_\alpha$. In other words, if $\erre[W_\alpha]$ and $\erre(W_\alpha)$ are, respectively, the coordinate ring of $W_\alpha$ and its field of fractions, there is an embedding $\erre(W_\alpha) \subset \mathfrak{K}(\alpha)$. As every semi-algebraic function satisfies a polynomial equation, the field $\mathfrak{K}(\alpha)$ is algebraic over $\erre(W_\alpha)$, it is its real closure with respect to the ordering induced by the embedding.

The convex hull of $\erre$ in $\mathfrak{K}(\alpha)$ is a valuation ring denoted by $\ocors_\leq$. As above, this valuation ring defines a valuation $v:\mathfrak{K}(\alpha)^* \freccia \Lambda$, where $\Lambda$ is an ordered abelian group. If $V$ has dimension $r$, the transcendence degree over $\erre$ of the field $\erre(W_\alpha)$ is at most $r$, hence, by \cite[cap. VI, thm. 3, cor. 1]{ZS2} and \cite[cap. VI, thm. 15]{ZS2}, the group $\Lambda$ has rank at most $r$.

As before the composition of the valuation $v$ with the quotient map $\Lambda \freccia \Lambda / \Lambda_{r-1}$, defines another valuation of $\mathfrak{K}(\alpha)$ that is real valued:
$$\overline{v}:\mathfrak{K}(\alpha)^* \freccia \erre$$

\begin{proposition}
Let $V$ be a semi-algebraic set, and let $\famil$ be a finite proper family of continuous positive semi-algebraic functions on $V$. If $g$ is a continuous positive semi-algebraic function, then there exist $A,B \in \erre_{>0}$ and $n \in \enne$ such that:
$$\forall x \in V : g(x) \leq A {\left(\sum_{f \in \famil}{f(x) + f^{-1}(x)}\right)}^n + B $$
\end{proposition}

\begin{proof}
Suppose, by contradiction, that the statement is false. Fix two increasing unbounded sequences $A_n, B_n \in \erre$, then for all $n \in \enne$ there exists a point $x_n \in V$ such that
$$ g(x_n) > A_n {\left(\sum_{f \in \famil}{f(x_n) + f^{-1}(x_n)}\right)}^n + B_n $$
There is no subsequence of $x_n$ contained in a compact subset of $V$, because $g(x_n)$ is unbounded. If we consider the compactification $\overline{V}_{\famil \cup \{g\}}$, then we can extract a subsequence $x_{n_k}$ converging to a point $[z] \in \partial_{\famil \cup \{g\}} V$. As $g$ grows faster than every $f \in \famil$ along the sequence $x_{k_n}$, a point $z \in C(\partial_{\famil \cup \{g\}} V)$ corresponding to $[z]$ has coordinates $z_g \neq 0$ and $z_f = 0$ for all $f \in \famil$.

The logarithmic limit set $\ameba_0(E_{\famil \cup \{g\}}(V))$ contains the point $(0,0,\dots,1)$, where the $1$ is in the coordinate corresponding to $g$. Then, by \cite[thm. 3.6]{A1}, there exists a sequence $y_k \in V$ such that $g(y_k) \tende \infty$ and $f(y_k)$ is bounded for all $f \in \famil$. But this is absurd because $\famil$ is a proper family.
\end{proof}

\begin{corollary}
For every $f \in \gfamil$ such that $y_f \neq 0$, we have $\overline{v}([f]) \neq 0$.

It follows from the previous proposition.
\end{corollary}

\begin{proposition}
Let $f,g \in \gfamil$, and suppose that $y_f \neq 0$. Then 
$$\frac{-\overline{v}([f])}{-\overline{v}([g])} = \frac{y_f}{y_g}$$
\end{proposition}

\begin{proof}
For every $\varepsilon >0$ there exists a neighborhood $U$ of $[y]$ in $\overline{V}_\gfamil$ such that
$$\forall x \in U : \left|\frac{\log_e(f(x))}{\log_e(g(x))} - \frac{y_f}{y_g}\right| < \varepsilon $$
$$\forall x \in U : {g(x)}^{\left(\frac{y_f}{y_g}-\varepsilon\right)} < f(x)  < {g(x)}^{\left(\frac{y_f}{y_g} + \varepsilon\right)}   $$
Hence, for every rational number $r < \frac{y_f}{y_g}$ there exists a neighborhood $U$ of $[y]$ in $\overline{V}_\gfamil$ such that
$$\forall x \in U : {g(x)}^{r} < f(x)  $$
and for every rational number $r > \frac{y_f}{y_g}$ there exists a neighborhood $U$ of $[y]$ in $\overline{V}_\gfamil$ such that
$$ \forall x \in U : f(x)  < {g(x)}^{r}$$
As all these neighborhoods $U$ are in $\alpha$, and the rational powers are semi-algebraic functions, these inequalities hold on the field $\mathfrak{K}(\alpha)$: for every rational number $r < \frac{y_f}{y_g}$ we have $[g]^r < [f]$ and for every rational number $ r > \frac{y_f}{y_g}$ we have $[f] < [g]^r$. Hence
$$\frac{-\overline{v}([f])}{-\overline{v}(g)} = \frac{y_f}{y_g}$$
\end{proof}

Now let $\overline{V} \subset {(\mathfrak{K}(\alpha))}^n$ be the extension of $V \subset \erre^n$ to the field $\mathfrak{K}(\alpha)$. For every $f \in \gfamil$, let $\overline{f}$ be the extension of $f$ to the field $\mathfrak{K}(\alpha)$.

\begin{proposition}      \label{prop:non-arch descr noncount}
There exists a point $x \in \overline{V}$ such that 
$${(-v(\overline{f}(x)))}_{f \in \gfamil} = y$$
\end{proposition}

\begin{proof}
If $x_1, \dots, x_n$ are the restriction to $V$ of the coordinate functions of $\erre^n$, then the point $([x_1], \dots, [x_n]) \in {(\mathfrak{K}(\alpha))}^n$ is in $\overline{V}$, because the coordinate functions restricted to $V$ satisfy the first order formulas satisfied by the points of $V$. Note also that $\overline{f}([x_1], \dots, [x_n]) = [f]$, hence, by the previous proposition, the point $([x_1], \dots, [x_n])$ is what we are searching for.
\end{proof}

\begin{theorem} \label{teo:non-arch descr}
Let $V \subset \erre^n$ be a semi-algebraic set of dimension $r$, and let $\gfamil$ be a proper family of positive continuous semi-algebraic functions on $V$. If $\effe = \erre((t^{\erre^r}))$, and $\overline{V}$ is the extension of $V$ to the field $\effe$, then $\Log_\gfamil(\overline{V}) = C(\partial_\gfamil V)$.
\end{theorem}

\begin{proof}
By the previous proposition, if $y \in C(\partial_\gfamil V)$, and $\alpha$ is an ultrafilter associated to $y$, there exists a real closed field $\mathfrak{K}(\alpha)$, with rank at most $r$, such that the extension of $V$ to $\mathfrak{K}(\alpha)$ has a point $x \in {(\mathfrak{K}(\alpha))}^n$ such that $\Log_\gfamil(x) = y$.

The conclusion follows from the fact that every real closed non-archimedean field of rank at most $r$ can be embedded in $\erre((t^{\erre^r}))$ (see \cite{Gl37}). Hence the image of the point $x$ in ${(\erre((t^{\erre^r}))}^n$ is in $\overline{V}$ and $\Log_\gfamil(x) = y$.
\end{proof}

\section{Character varieties over real closed fields}   \label{sez:realchar}

We give a description of the variety of characters of representations of a finitely generated group $\Gamma$ in $SL_n(\cappa)$. The space of all representations $\homom(\Gamma, SL_n(\cappa))$ is an affine algebraic set (see subsection \ref{subsez:rappr}). If $\cappa = \ci$, it follows from the theory in \cite{MFK94} and \cite{Pr76} that the set of all characters of representations in $\homom(\Gamma, SL_n(\ci))$ is an affine algebraic set $\charat(\Gamma, SL_n(\ci))$ with a natural map
$$t:\homom(\Gamma, SL_n(\ci)) \freccia \charat(\Gamma, SL_n(\ci))$$ 
sending each representation in its character (see subsection \ref{subsez:chars}). We need similar results when $\cappa = \erre$. In this case, if we denote by $\charat(\Gamma, SL_n(\erre))$ the real part of $\charat(\Gamma, SL_n(\ci))$, the map
$$t:\homom(\Gamma, SL_n(\erre)) \freccia \charat(\Gamma, SL_n(\erre))$$
is not surjective in general, hence the affine algebraic set $\charat(\Gamma, SL_n(\erre))$ is not in bijection with the set of all characters of representations in $\homom(\Gamma, SL_n(\erre))$. We prove that the image of $t$ is closed, identifying the set of characters with a closed semi-algebraic subset $\overline{\charat}(\Gamma, SL_n(\erre))$, and that the image through $t$ of every closed (open) conjugation-invariant semi-algebraic subset of $\homom(\Gamma, SL_n(\erre))$ is closed (open) in $\overline{\charat}(\Gamma, SL_n(\erre))$ (theorem \ref{teo:real char} and corollary \ref{corol:real char}).

\subsection{Representation varieties}     \label{subsez:rappr}

Let $\Gamma$ be a group and $\cappa$ a field of characteristic $0$. An \nuovo{representation} of $\Gamma$ is a group homomorphism $\rho:\Gamma \freccia GL_n(\cappa)$ or $\rho:\Gamma \freccia PGL_n(\cappa)$.

A representation $\rho$ is \nuovo{absolutely irreducible} if it is irreducible as a representation in $GL_n(\effe)$ (or $PGL_n(\effe)$) where $\effe$ is the algebraic closure of $\cappa$, else it is \nuovo{absolutely reducible}.

The \nuovo{character} of a representation $\rho$ is the function
$$\chi_\rho:\Gamma \ni \gamma \freccia \trace(\rho(\gamma)) \in \cappa$$
By the conjugation-invariance of the trace, two conjugated representations have the same character. A sort of converse holds: let $\rho, \rho'$ be two representations, and suppose that $\rho$ is absolutely irreducible. Then they are conjugated if and only if they have the same character. See also \cite[thm. 6.12]{Na00} for a more general statement.

In the following $\Gamma$ is assumed to be a finitely generated group. Let $G$ an affine algebraic group, and let $G(\cappa)$ denote the set of $\cappa$-points of $G$. There exists an affine $\qu$-algebraic scheme $\homom(\Gamma, G)$ such that for every field $\cappa$, the set of $\cappa$-points $\homom(\Gamma, G(\cappa))$ is in natural bijection with the set of all representations of $\Gamma$ in $G(\cappa)$.

Now suppose that $G$ is one of the groups $SL_n$ or $SL^{\pm}_n$ (the group of matrices whose determinant is $\pm 1$). Then the set of $\cappa$-points $G(\cappa)$ is the group $SL_n(\cappa)$ or $SL^{\pm}_n(\cappa)$.
The $\qu$-algebraic group $PGL_n$ acts on $G$ by conjugation, and this action induces an action on $\homom(\Gamma, G)$. Every $\gamma \in \Gamma$ defines a polynomial function
$$\tau_\gamma: \homom(\Gamma,G(\cappa)) \ni \rho \freccia \chi_\rho(\gamma) \in \cappa$$
these functions comes from functions $\tau_\gamma$ in the ring of coordinates of $\homom(\Gamma,G))$, and they will be called \nuovo{trace functions}. The trace functions are invariant for the action of $PGL_n$

There exists a closed subscheme ${\homom(\Gamma,G)}_{a.r.r.}$ of $\homom(\Gamma, G)$ whose set of $\cappa$ points ${\homom(\Gamma,G(\cappa))}_{a.r.r.}$ is the subset of all absolutely reducible representations (see \cite{Na00}). We define also ${\homom(\Gamma,G)}_{a.i.r}$ as the complement of ${\homom(\Gamma,G)}_{a.r.r.}$, the set of absolutely irreducible representations, an open subscheme.

Let $V \subset \homom(\Gamma,G(\cappa))$ be an irreducible algebraic subset. We denote by $\anello(V)$ its ring of coordinates, and by $\cappa(V)$ its field of fractions. A point $\rho \in V$ is a representation $\rho:\Gamma \freccia G(\cappa)$. We write $\rho(\gamma) = \left( a_{i,j}^\gamma(\rho) \right)$, where $a_{i,j}^\gamma \in \anello(V)$ as they are restriction of polynomial functions in $\homom(\Gamma,G(\cappa))$. Note that for every $\rho \in V$, $\det\left( a_{i,j}^\gamma(\rho) \right) \neq 0$, hence the element $\det\left( a_{i,j}^\gamma \right)$ is invertible in $\cappa(V)$. The map:
$$\Gamma \ni \gamma \freccia \left( a_{i,j}^\gamma \right) \in GL_n(\cappa(V))$$
is the canonical representation in $G(\cappa(V))$, that will be denoted by $\rappr_V$. The character of $\rappr_V$ is the function $\chi_{\rappr_V}(\gamma) = \sum_i a_{i,i}^\gamma = \tau_\gamma$.

Consider the action by conjugation of $PGL_n(\cappa)$ on $\homom(\Gamma,G(\cappa))$.

\begin{proposition}
Let $V \subset \homom(\Gamma,G(\cappa))$ be an irreducible component. Then $V$ is invariant for the action of $PGL_n(\cappa)$ by conjugation.
\end{proposition}

\begin{proof}
The same statement for $n=2$ is proved in \cite[prop. 1.1.1]{CS83}. Consider the set $V \times PGL_n(\cappa)$, this is the product of two irreducible affine algebraic sets, hence it is an irreducible affine algebraic set. The map
$$f:V \times PGL_n(\cappa) \ni (\rho,A) \freccia \adjoint(A)(\rho) \in \homom(\Gamma,G(\cappa))$$
is a regular map, hence the Zariski closure of the image $f(V \times PGL_n(\cappa))$ is an irreducible affine algebraic set, hence it is contained in an irreducible component of $\homom(\Gamma,G(\cappa))$. But as $V = f(V \times \{\ident\})$, then $f(V \times PGL_n(\cappa)) \subset V$.
\end{proof}

\subsection{Character varieties}       \label{subsez:chars}

As before, suppose that $G$ is one of the groups $SL_n$ or $SL^{\pm}_n$. We denote by $A$ the ring of coordinates of $\homom(\Gamma,G)$, and by $A_0$ the subring of invariant functions for the action of $PGL_n$. As $PGL_n$ is reductive, by \cite[Chap. 1, thm. 1.1]{MFK94}, the ring $A_0$ is finitely generated as a $\qu$-algebra. Note that the trace functions $\tau_\gamma$ belong to $A_0$.

\begin{proposition}
There exists a finite set $C \subset \Gamma$ such that the functions ${\{\tau_{\gamma}\}}_{\gamma \in C}$ generate $A_0$.
\end{proposition}

\begin{proof}
The case of a free group $\Gamma = \ze^{(m)}$, is studied in \cite{Pr76}. There it is proven that $A_0$ is generated, as a $\qu$-algebra, by the set ${\{\tau_{\gamma}\}}_{\gamma \in \Gamma}$. Hence there exists a finite set $C \subset \Gamma$ such that the functions ${\{\tau_{\gamma}\}}_{\gamma \in C}$ generate the ring $A_0$. Some of such $C$ are described explicitly in \cite{Pr76}. For example let $z_1, \dots, z_m$ be a free set of generators of $\Gamma$, and let $C \subset \Gamma$ be the finite set of all non-commutative monomials in $z_1, \dots, z_m$ of degree less than $2^n$.

For the general case, let $\Gamma$ be a finitely generated group and let $P:\ze^{(m)} \freccia \Gamma$ be a presentation. The map $P$ induces an injective map $P^*:\homom(\Gamma,G) \freccia \homom(\ze^{(m)},G)$, identifying $\homom(\Gamma,G)$ with an invariant closed subscheme. If $A$ is the ring of coordinates of $\homom(\Gamma,G)$ and $R$ is the ring of coordinates of $\homom(\ze^{(m)},G)$, the natural map $P_*:R \freccia A$ is a surjective ring homomorphism, commuting with the dual action of $PGL_n$. Let $A_0 \subset A$  and $R_0 \subset R$ be the subrings of invariant functions. Then, as  $PGL_n$ is reductive, $p(R_0) = A_0$ (see the discussion after \cite[Chap. 1, def. 1.5]{MFK94}). As $R_0$ is generated by the trace functions, so is $A_0$.
\end{proof}

Let $C \subset \Gamma$ be a finite set such that the functions ${\{\tau_{\gamma}\}}_{\gamma \in C}$ generate the ring $A_0$. Let $\cappa$ be an algebraically closed field of characteristic $0$, and consider the map
$$t: \homom(\Gamma,G(\cappa)) \ni \rho \freccia {\tau_\gamma(\rho)}_{\gamma \in C} \in \cappa^{\card(C)}$$
We will denote by $\charat(\Gamma, G(\cappa))$ the Zariski closure of the image of this map, an affine $\qu$-algebraic set whose ring of coordinates is isomorphic to $A_0$.

The map $t$ is dual to the inclusion map $A_0 \freccia A$, hence it is identified with the semi-geometric quotient $\homom(\Gamma,G) = \spec(A) \freccia \spec A_0$ as in \cite[Chap. 1, thm. 1.1]{MFK94}. As this semi-geometric quotient is surjective, the map $t: \homom(\Gamma,G(\cappa)) \freccia \charat(\Gamma, G(\cappa))$ is surjective. We will write $\charat(\Gamma, G) = \spec(A_0)$.

If $C' \subset \Gamma$ is another finite set of generators, the pair $(\charat(\Gamma, G), t)$ defined by $C'$ is isomorphic to the previous one, hence this construction does not depend on the choices.

The functions ${\{\tau_{\gamma}\}}_{\gamma \in C}$ determine the values of all the trace functions ${\{\tau_{\gamma}\}}_{\gamma \in \Gamma}$, hence, if $\rho$ is a representation, the point $t(\rho)$ determines the character $\chi_\rho$. Hence the points of $\charat(\Gamma, G(\cappa))$ are in natural bijection with the characters of the representations in $\homom(\Gamma,G(\cappa))$, and for this reason the affine $\qu$-algebraic set $\charat(\Gamma, G(\cappa))$ will be called \nuovo{varieties of characters}. For the properties of the map $t: \homom(\Gamma,G(\cappa)) \freccia \charat(\Gamma, G(\cappa))$, see \cite[Chap. 1, thm. 1.1]{MFK94}.

To avoid confusion, in the following we will denote by $\tau_\gamma$ the trace function relative to the element $\gamma \in \Gamma$ when considered as a function on the representation variety, and by $I_\gamma$ the same trace function when considered as a function on the character variety.

Consider the invariant subsets ${\homom(\Gamma,G(\cappa))}_{a.r.r.}$ and ${\homom(\Gamma,G(\cappa))}_{a.i.r}$ of, respectively, absolutely reducible and absolutely irreducible representations. As ${\homom(\Gamma,G(\cappa))}_{a.r.r.}$ is closed, its image through $t$ is closed, and will be denoted by ${\charat(\Gamma,G(\cappa))}_{a.r.r.}$. As ${\homom(\Gamma,G(\cappa))}_{a.i.r}$ is open, its images through $t$ is open, and will be denoted by ${\charat(\Gamma,G(\cappa))}_{a.i.r.}$. As two representation with the same character are conjugated, if at least one of them is irreducible, ${\charat(\Gamma,G(\cappa))}_{a.r.r.}$ and ${\charat(\Gamma,G(\cappa))}_{a.i.r.}$ are disjoint. They are the sets of $\cappa$-points of algebraic schemes ${\charat(\Gamma,G)}_{a.r.r.}$ and ${\charat(\Gamma,G)}_{a.i.r.}$.

Consider the restriction of $t$ to ${\homom(\Gamma,G)}_{a.i.r.}$:
$$t_{a.i.r.}: {\homom(\Gamma,G)}_{a.i.r.} \freccia {\charat(\Gamma,G)}_{a.i.r.}$$
This is a semi-geometric quotient, it is submersive, and the geometric fibers of $t$ are precisely the orbits of the geometric points of ${\homom(\Gamma,G)}_{a.i.r.}$, over an algebraically closed field $\cappa$ of characteristic $0$. Hence this is a geometric quotient, ${\homom(\Gamma,G)}_{a.i.r.} \subset \homom(\Gamma,G)^s(\mbox{Pre})$, and the set-theoretical quotient ${\homom(\Gamma,G(\cappa))}_{a.i.r.} / PGL_n(\cappa)$ is in natural bijection with ${\charat(\Gamma,G(\cappa))}_{a.i.r.}$.

Actually the action of $PGL_n$ on the open subscheme ${\homom(\Gamma,G)}_{a.i.r.}$ is free (see \cite[corol. 6.5]{Na00}) and ${\homom(\Gamma,G)}_{a.i.r.} \subset \homom(\Gamma,G)$ is precisely the subset of \nuovo{properly stable points} for the action of $PGL_n$ with respect to the canonical linearization of the trivial line bundle (see \cite[Chap. 1, def. 1.8]{MFK94} and \cite[rem. 6.6]{Na00}).

\subsection{Real closed case}

We need a similar construction for a real closed field $\effe$. If $G = SL_n \mbox{ or } SL^{\pm}_n$, the set of characters of representations $\rho:\Gamma \freccia G(\effe)$ is not an affine algebraic set. Here we will show that this set is a closed semi-algebraic set, and that the map $t:\homom(\Gamma,G(\effe)) \freccia \charat(\Gamma,G(\effe))$ has properties similar to the properties it has in the algebraically closed case.

Let $\cappa = \effe[i]$, the algebraic closure of $\effe$. As the representation varieties are defined over $\qu$, if $\homom(\Gamma,G(\cappa)) \subset \cappa^m$, we have $\homom(\Gamma,G(\effe)) = \homom(\Gamma,G(\cappa)) \cap \effe^n$, and if $\charat(\Gamma,G(\cappa)) \subset \cappa^s$, we have $\charat(\Gamma,G(\effe)) = \charat(\Gamma,G(\cappa)) \cap \effe^s$.

The map $t:\homom(\Gamma,G(\cappa)) \freccia \charat(\Gamma,G(\cappa))$ is defined over $\qu$, hence $t(\homom(\Gamma,G(\effe)) \subset \charat(\Gamma,G(\effe))$. Anyway $t(\homom(\Gamma,G(\effe)))$ is not in general the whole $\charat(\Gamma,G(\effe))$. For example an irreducible representation of $\Gamma$ in $SU_2(\ci)$ has real character, but it is not conjugated to a representation in $SL_2(\erre)$ (see \cite[prop. III.1.1]{MS84} and the discussion for details). Hence the $\effe$-algebraic set $\charat(\Gamma,G(\effe))$ is not in a natural bijection with the set of characters of representations in $\homom(\Gamma,G(\effe))$. We will denote by $\overline{\charat}(\Gamma,G(\effe))$ the image of $t_{|\homom(\Gamma,G(\effe))}$, the actual set of characters of representations in $\homom(\Gamma,G(\effe))$.

In the following we will consider $\effe^n$ as a topological space with the topology inherited from the order of $\effe$, hence the words closed and open refers to this topology. We will say Zariski closed and Zariski open if we want to refer to the Zariski topology.

\begin{theorem}  \label{teo:real char}
Let $R \subset \homom(\Gamma,G(\effe)) \subset \effe^m$ be a closed semi-algebraic set that is invariant for the action of $PGL_n(\effe)$. Then the image $t(R)$ under the semi-geometric quotient map $t$ is a closed semi-algebraic subset of $\effe^s$.
\end{theorem}

\begin{proof}
The idea of the proof is similar to the one of \cite[prop. 1.4.4]{CS83}, but here we deal with semi-algebraic sets, instead of algebraic sets over an algebraically closed field.

The set $t(R)$ is semi-algebraic as it is the image of a semi-algebraic set via a polynomial map. To show that it is closed, let $x_0$ be a point in the closure of $t(R)$. Suppose, by contradiction, that $x_0$ is not in $t(R)$.

By the curve-selection lemma (see \cite[Chap. 2, thm. 2.5.5]{BCR98}) there exists a semi-algebraic map $f:[0,\varepsilon] \freccia \effe^{\card(C)}$ such that $f(0) = x_0$ and $F = f((0,\varepsilon]) \subset \chi(R)$. The set $F$ is a semi-algebraic curve.

Now we construct a semi-algebraic section $s:F \freccia R$ such that for every $x \in F$, $t(s(x)) = x$. The set $D = t^{-1}(F)$ is a semi-algebraic subset of $R$. The map $t_{|D}: D \freccia F$ induces a definable equivalence relation on $D$: $x \sim y \Leftrightarrow t(x) = t(y)$. By the existence of definable choice functions (see \cite[Chap. 6, 1.2-1.3]{Dr}) there exists a semi-algebraic map $h: D \freccia D$ such that $h(x) = h(y) \Leftrightarrow t(x) = t(y)$, and a semi-algebraic map $k:h(D) \freccia F$ such that $t = k \circ h$. The map $k$ is bijective, and its inverse $k^{-1} = s$ is the searched section.

Consider the image $s(F) \subset R \subset \effe^m$. If it is bounded in $\effe^m$, then the closure $S$ of $s(F)$ is contained in $R$, and $t(S)$ is closed (image of a closed, bounded set, see \cite[Chap. 6, 1.10]{Dr}) hence it contains $x_0$, a contradiction. Hence the image $s(F)$ is unbounded.

We embed $\effe^m$ and $\effe^s$ in the projective spaces $\effe\pro^m$ and $\erre\pro^s$. We denote by $\overline{R}$ the closure of $R$ in $\effe\pro^m$,  by $\overline{t(R)}$ the closure of $t(R)$ in $\erre\pro^s$ and by $\overline{t} : \overline{R} \freccia \overline{t(R)}$ the extension of the map $t$, which exists because $t$ is a polynomial map.

If we see $\effe\pro^m$ as a closed and bounded algebraic subset of some $\effe^M$, the image $s(F)$ is bounded in $\effe^M$, its closure $S$ for the order topology is contained in $\overline{R}$ and $\overline{t}(S)$ is, as before, closed, hence there is a point $y_0 \in S \subset \overline{R}$ such that $\overline{t}(y_0) = x_0$.

Let $E$ be the Zariski closure of $s(F)$ in $\effe\pro^m$. Up to restricting $f$ to a smaller interval $[0,\varepsilon']$, we can suppose that $E$ is an irreducible algebraic curve (see \cite[Chap. 2, prop. 2.8.2]{BCR98}) containing $y_0$. Let $E^\cappa$ be the extension of $E$ to $\cappa\pro^m$. Let $\effe(E)$ be the field of fractions of the curves $E$ and $E^\cappa$, with coefficients in $\effe$, and $\cappa(E^\cappa)$ the field of fractions of $E^\cappa$ with coefficients in $\cappa$.

Let $\widetilde{E^\cappa}$ be the regular projective model of $E^\cappa$, i.e. a regular projective curve with a morphism $i:\widetilde{E^\cappa} \freccia E^\cappa$ that is a birational isomorphism (see \cite[\numero 7A, thm.~7.5]{Mu76}). We denote by $\widetilde{E}$ the inverse image $i^{-1}(E)$. Hence there is an isomorphism $\cappa(E^\cappa) \simeq \cappa(\widetilde{E^\cappa})$. Let $\widetilde{y_0}$ be an element of $i^{-1}(y_0)$. As $\widetilde{y_0}$ is a regular point, the local ring $\ocors_{\widetilde{y_0},\widetilde{E^\cappa}}$ is an UFD, and, as it has dimension $1$, its maximal ideal $m_{\widetilde{y_0},\widetilde{E^\cappa}}$ is the unique prime ideal. Every irreducible element $\pi \in \ocors_{\widetilde{y_0},\widetilde{E^\cappa}}$ generate a prime ideal, hence $(\pi) = m_{\widetilde{y_0},\widetilde{E^\cappa}}$, hence all irreducible elements are associated. Let $v^\cappa:\cappa(E^\cappa)^* \freccia \ze$ be the $\pi$-adic valuation. We consider the restricted valuation $v:\effe(E)^* \freccia \ze$, and we denote $\ocors \subset \effe(E)$ the valuation ring.

Let $\rappr_E:\Gamma \freccia GL_n(\effe(E))$, the canonical representation. We want to show that for every $\gamma \in \Gamma$, $\chi_{\rappr_E}(\gamma) \subset \ocors$, i.e. that all functions $\tau_\gamma$ don't have a pole in $\widetilde{y_0}$. It is enough to show this for the functions ${\{\tau_\gamma\}}_{\gamma \in C}$, as all other functions $\tau_\gamma$ are polynomials in these ones. As $t(y_0) = x_0$, a point in $\effe^m$, we know that the functions ${\{\tau_\gamma\}}_{\gamma \in C}$ have a finite value in $y_0$, hence they don't have a pole in $\widetilde{y_0}$.

Hence $\chi_{\rappr_E}:\Gamma \freccia \ocors$. This implies that this representation is conjugated, via $PGL_n(\effe(E))$, with a representation $\rappr':\Gamma \freccia GL_n(\ocors)$ (see the note after \cite[prop. II.3.17]{MS84}). Hence there exists a matrix $M = (m_{i,j}) \in GL_n(\effe(E))$ such that for all $\gamma \in \Gamma$, $ \rappr'(\gamma) = M \rappr_{E}(\gamma) M^{-1}$.

Let $U = E \cap \homom(\Gamma,G(\effe)) \setminus \{ \mbox{ poles of the functions } m_{i,j} \}$, a Zariski open subset of $E$. We define the map 
$$d:U \ni \rho \freccia M(\rho) \rho {M(\rho)}^{-1} \in \homom(\Gamma,G(\effe))$$
As for every $\rho \in U$, $d(\rho)$ is conjugated to $\rho$, then $t(\rho) = t(d(\rho))$. Moreover, for all $\rho \in U \cap R$, $d(\rho) \in R$, as $R$ is invariant. The map $d$ extends to a rational map $\widetilde{d}:\widetilde{E^\cappa} \freccia \overline{\homom(\Gamma,G(\cappa))}$, that is a morphism as $\widetilde{E^\cappa}$ is regular. Then $t \circ \widetilde{d}_{|\widetilde{E}} = t \circ i_{|\widetilde{E}}$, as this holds on an open subset. As the representation $\rappr'$ takes values in $GL_n(\ocors)$, then $\widetilde{d}(y_0) \in \homom(\Gamma,G(\effe))$, and as $d(U \cap R) \subset R$, and $y_0$ is in the closure of $U$, then $\widetilde{d}(y_0) \in R$. We have $t(\widetilde{d}(y_0)) = t(i_{|\widetilde{E}}(y_0)) = x_0$, hence $x_0$ is in the image $t(R)$, a contradiction.
\end{proof}

\begin{corollary}        \label{corol:real char}
The surjective map
$$t:\homom(\Gamma,G(\effe)) \freccia \overline{\charat}(\Gamma,G(\effe))$$
has the following properties:
\begin{enumerate}
\item The set $\overline{\charat}(\Gamma,G(\effe))$ is a closed semi-algebraic set in natural bijection with the set of characters of representations in $\homom(\Gamma,G(\effe))$.
\item If $R$ is an invariant closed semi-algebraic subset of $\homom(\Gamma,G(\effe))$, then $t(R)$ is a closed semi-algebraic set. If $R$ is an invariant open semi-algebraic subset of $\homom(\Gamma,G(\effe))$, then $t(R)$ is a semi-algebraic set that is open in $\overline{\charat}(\Gamma,G(\effe))$.
\item The set ${\overline{\charat}(\Gamma,G(\effe))}_{a.r.r.} = t({\homom(\Gamma,G(\effe))}_{a.r.r.})$ is a closed semi-algebraic set and the set $\overline{\charat}(\Gamma,G(\effe)) = t({\homom(\Gamma,G(\effe))}_{a.i.r.})$ is a semi-algebraic set that is open in $\overline{\charat}(\Gamma,G(\effe))$.
\item The set ${\overline{\charat}(\Gamma,G(\effe))}_{a.i.r.}$ is in natural bijection with the set theoretical quotient ${\homom(\Gamma,G(\effe))}_{a.i.r.} / PGL_n(\effe)$.
\end{enumerate}
\end{corollary}

\begin{proof}
Everything follows from the previous theorem and from \cite[Chap. 1, thm. 1.1]{MFK94}.
\end{proof}

\section{Spaces of projective structures}        \label{sez:projective}

\subsection{Deformation spaces}

We refer to \cite{G} for the definitions related to the theory of geometric structures on manifolds. In the following we will consider the geometry $\erre\pro^n = (\erre\pro^n, PGL_{n+1}(\erre))$ the \nuovo{projective geometry}. A \nuovo{projective structure} is an $\erre\pro^n$-structure, and a \nuovo{projective manifold} is a $\erre\pro^n$-manifold.

Let $S$ be a manifold. A \nuovo{marked} $\erre\pro^n$-\nuovo{structure} on $S$ is a pair $(M,\phi)$, where $M$ is an $\erre\pro^n$-manifold and $\phi:S \freccia M$ is a diffeomorphism. The diffeomorphism $\phi$ induces an $\erre\pro^n$-structure on $S$. Two marked $\erre\pro^n$-structures $(M,\phi)$, $(N,\psi)$ on $S$ are \nuovo{isotopic} if there is an $\erre\pro^n$-map $h:M \freccia N$ such that $\psi$ is isotopic to $h\circ \phi$. Note that $h$ is necessarily an isomorphism.

We choose a base point $s_0 \in S$ and a universal covering space $\widetilde{S} \freccia S$. A \nuovo{based} $\erre\pro^n$-\nuovo{structure} on $S$ is a triple $(M,\phi,\psi)$ where $M$ is an $\erre\pro^n$-manifold, $\phi:S \freccia M$ is a diffeomorphism and $\psi$ is an $\erre\pro^n$-germ at $\phi(s_0)$. The diffeomorphism $\phi$ induces an $\erre\pro^n$-structure on $S$. The germ $\psi$ determines a developing pair $(D,h)$ for $M$, and this developing pair induces, via the diffeomorphism $\phi$, a developing pair $(f,\rho)$ for the $\erre\pro^n$-structure on $S$, such that $\rho:\pi_1(S,s_0) \freccia PGL_n(\erre)$ is a representation, and $f:\widetilde{S} \freccia X$ is a $\rho$-equivariant local diffeomorphism. Vice versa every such pair $(f,\rho)$ determines a based $\erre\pro^n$-structure on $S$.

We say that two based $\erre\pro^n$-structures $(f,\rho)$ and $(f',\rho')$ are \nuovo{isotopic} if $\rho = \rho'$ and there exists a diffeomorphism $h:(S,s_0) \freccia (S,s_0)$, isotopic to the identity, such that $f' = f \circ \widetilde{h}$, where $\widetilde{h}$ is the lift of $h$ to $\widetilde{S}$.

We consider the algebraic set $\homom(\pi_1(S,s_0),PGL_n(\erre))$ with the topology induced by the order topology of $\erre$, and the set $C^\infty(\widetilde{S},X)$ of smooth maps $\widetilde{S} \freccia X$ with the $C^\infty$ topology.

We define the deformation set of based $\erre\pro^n$-structures:
$$\mathcal{D}_{\erre\pro^n}'(S) = \{ (f,\rho) \in C^\infty(\widetilde{S},X) \times \homom(\pi_1(S,s_0),PGL_n(\erre)) \ |\ $$ 
$$f \mbox{ is a }\rho\mbox{-equivariant local diffeomorphism} \}$$

This set inherits the subspace topology. We denote by $\diff(S,s_0)$ the group of all diffeomorphisms $S \freccia S$ fixing $s_0$, and by $\diff_0(S,s_0)$ the subgroup of all diffeomorphisms fixing $s_0$ and isotopic to the identity. The group $\diff_0(S,s_0)$ acts properly and freely on $\mathcal{D}_{\erre\pro^n}'(S)$. We denote by $\mathcal{D}_{\erre\pro^n}(S)$ the quotient by this action, the set of isotopy classes of based $\erre\pro^n$-structures:
$$\mathcal{D}_{\erre\pro^n}(S) = \mathcal{D}_{\erre\pro^n}'(S) / \diff_0(S,s_0)$$
this set is endowed with the quotient topology. The group $PGL_n(\erre)$ acts on $\mathcal{D}_{\erre\pro^n}'(S)$ by composition on $f$ and by conjugation on $\rho$, and this action passes to the quotient $\mathcal{D}_{\erre\pro^n}(S)$. We will denote the quotient by
$$\mathcal{T}_{\erre\pro^n}(S) = \mathcal{D}_{\erre\pro}(S) / PGL_n(\erre)$$
This set is endowed with the quotient topology. It is in natural bijection with the set of marked $\erre\pro^n$-structures up to isotopy.

The holonomy map
$$\hol_\mathcal{D}':\mathcal{D}_{\erre\pro}'(S) \ni (f,\rho) \freccia \rho \in \homom(\pi_1(S,s_0),PGL_n(\erre))$$
is continuous and it is invariant under the action of $\diff_0(S,s_0)$, hence it defines a continuous map
$$\hol_{\mathcal{D}}:\mathcal{D}_{\erre\pro}(S) \freccia \homom(\pi_1(S,s_0),PGL_n(\erre))$$
The group $PGL_n(\erre)$ acts on $\homom(\pi_1(S,s_0),PGL_n(\erre))$ by conjugation, and on $\mathcal{D}_{\erre\pro}(S)$ as said. The map $\hol_{\mathcal{D}}$ is equivariant with respect to these $PGL_n(\erre)$-actions, hence it induces a continuous map
$$\hol_{\mathcal{T}}:\mathcal{T}_{\erre\pro^n}(S) \freccia \homom(\pi_1(S,s_0),PGL_n(\erre)) / PGL_n(\erre)$$

By the deformation theorem (see \cite[5.3.1]{Th}, \cite{Lo84} or \cite{CEG87}), the map $\hol_{\mathcal{D}}'$ is open and the map $\hol_{\mathcal{D}}$ is a local homeomorphism. To prove that also $\hol_{\mathcal{T}}$ is a local homeomorphism we need additional hypotheses.

As the group $PGL_n(\erre)$ is a reductive group, the set of \nuovo{stable points} ${\homom(\pi_1(S,s_0),PGL_n)}^{s} \subset \homom(\pi_1(S,s_0),PGL_n)$ (see \cite[Chap. 1, def. 1.7]{MFK94}) is a Zariski open invariant subset upon which the action of $PGL_n$ is proper. We denote ${\mathcal{T}_{\erre\pro^n}(S)}^{s} = \hol^{-1}({\homom(\pi_1(S,s_0),PGL_n(\erre))}^{s})$. Then the map:
$$\hol_{\mathcal{T}}: {\mathcal{T}_{\erre\pro^n}(S)}^{s} \freccia {\homom(\pi_1(S,s_0),PGL_n(\erre))}^{s} / PGL_n(\erre)$$
is a local homeomorphism (see \cite[cor. 3.2]{G}).

\subsection{Convex projective structures}      \label{subsez:projective geometry}

An \nuovo{affine space} in $\erre\pro^n$ is the complement of a projective hyperplane. A set $\Omega \subset \erre\pro^n$ is \nuovo{convex} if it is contained in some affine space and its intersection with every projective line is connected. A convex set is \nuovo{properly convex} if its closure $\overline{\Omega}$ is contained in an affine space. A properly convex set $\Omega \subset \erre\pro^n$ is \nuovo{strictly convex} if its boundary $\partial \Omega$ does not contain any segment.

A \nuovo{convex projective manifold} is a projective manifold isomorphic to $\Omega / \Gamma$, where $\Omega \subset \erre\pro^n$ is an open properly convex domain and $\Gamma \subset PGL_{n+1}(\erre)$ is a discrete group acting properly and freely on $\Omega$. A \nuovo{strictly convex projective manifold} is a convex projective manifold $\Omega / \Gamma$, where $\Omega$ is strictly convex. In other words, a projective structure is strictly convex if an only if the developing map is injective, with image a strictly convex open subset of $\erre\pro^n$.

For example, $\iper^n \subset \erre\pro^n$ is an open, strictly convex set, $O^+(1,n) \subset PGL_{n+1}(\erre)$, hence every complete hyperbolic manifold is a strictly convex projective manifold.

A subgroup $\Gamma \subset PGL_n(\erre)$ \nuovo{divides} an open properly convex set $\Omega$ if $\Gamma$ is discrete, $\Omega$ is invariant for $\Gamma$ and the quotient $\Omega / \Gamma$ is compact. If such a group exists, $\Gamma$ is called a \nuovo{divisible convex set}. If $\Gamma$ is torsion-free, the action is free, hence the quotient $\Omega / \Gamma$ is a closed manifold, with a convex projective structure.

Let $\Gamma$ be a group. The \nuovo{virtual center} of $\Gamma$ is the subgroup of all elements of $\Gamma$ whose centralizer has finite index in $\Gamma$. The virtual center of $\Gamma$ is trivial if and only if every subgroup with finite index in $\Gamma$ has trivial center. A subgroup $\Gamma \subset PGL_{n+1}(\erre)$ is \nuovo{strongly irreducible} if and only if all subgroups of finite index of $\Gamma$ are irreducible. A properly convex set $\Omega \subset \erre\pro^n$ is said to be \nuovo{reducible} if there exists a direct-sum decomposition $\erre^{n+1} = V \oplus W$ such that the cone $\overline{\Omega} = \pi^{-1}(\Omega) \cup \{0\} \subset \erre^{n+1}$ is the direct sum $C + D$ of two convex cones $C \subset V$, $D \subset W$, otherwise it is said to be \nuovo{irreducible}.

\begin{proposition}
Let $\Gamma \subset PGL_{n+1}(\erre)$ be a group dividing an open properly convex set $\Omega$. The following are equivalent:
\begin{enumerate}
 \item $\Gamma$ has trivial virtual center.
 \item $\Gamma$ is strongly irreducible.
 \item $\Omega$ is irreducible.
\end{enumerate}
\end{proposition}

\begin{proof}
See \cite{Ve70}, \cite{Be0} and \cite{Be3}.
\end{proof}

As a corollary, if $M$ is a complete hyperbolic manifold, the fundamental group of $M$ has trivial virtual center, as it is isomorphic to a subgroup of $O^+(1,n)$ dividing $\iper^n$, that is irreducible.

\begin{proposition}
If $\Gamma \subset PGL_{n+1}(\erre)$ divides an open properly convex set $\Omega$. Then $\Omega$ is strictly convex if and only if the group $\Gamma$ is Gromov hyperbolic.
\end{proposition}

\begin{proof}
See \cite{Be1}.
\end{proof}

As this property depends only on the abstract group $\Gamma$, if the fundamental group of a manifold $M$ is Gromov hyperbolic, then all convex projective structures on $M$ are strictly convex. For example every closed hyperbolic $n$-manifold is a strictly convex projective manifold.

Let $SL^{\pm}_{n+1}(\erre) \subset GL_{n+1}(\erre)$ be the subgroup of matrices with determinant $\pm 1$. Then $PGL_{n+1} = SL^{\pm}_{n+1}(\erre) / \{\pm \ident\}$.

If $\gamma \in PGL_{n+1}(\erre)$, let $\overline{\gamma} \in SL^{\pm}_{n+1}(\erre)$ be a lift. Let $\lambda_1(\gamma), \dots, \lambda_{n+1}(\gamma)$ be its complex eigenvalues, ordered such that $|\lambda_1(\gamma)| \geq |\lambda_2(\gamma)| \geq \dots \geq |\lambda_{n+1}(\gamma)|$. The element $\gamma$ is said to be \nuovo{proximal} if $|\lambda_1(\gamma)| > |\lambda_2(\gamma)|$. In this case $\lambda_1(\gamma)$ is real, and its eigenvector corresponds to the unique attracting fixed point $x_g \in \erre\pro^n$ of $g$.

\begin{proposition}
Let $\Gamma \subset PGL_{n+1}(\erre)$ be a torsion-free group dividing a strictly convex set $\Omega$. Then every element $\gamma \in \Gamma$ is proximal. In particular $\gamma^{-1}$ is also proximal, hence the eigenvector $\lambda_{n+1}(\gamma)$ is real. Moreover, if $\overline{\gamma} \in SL^{\pm}_{n+1}(\erre)$ is a lift of $\gamma$, then $\lambda_1(\overline{\gamma})$ and $\lambda_{n+1}(\overline{\gamma})$ have the same sign.
\end{proposition}

\begin{proof}
See \cite{Be1}.
\end{proof}

The point $y_{\gamma} = x_{\gamma^{-1}}$ is the unique repelling fixed point of $\gamma$. The points $x_\gamma, y_\gamma$ are in $\partial \Omega$, and the segment $[x_g, y_g]$ is the unique invariant geodesic of $\gamma$ in $\Omega$. The image of $[x_\gamma,y_\gamma]$ in $\Omega / \Gamma$ is the unique geodesic in the free-homotopy class of $\gamma$. Moreover, $\Omega / \Gamma$ does not contain any closed homotopically trivial geodesic.

\begin{corollary}  \label{corol:lift}
The set $\pi^{-1}(\Omega) \subset \erre^{n+1}$ is the union of two convex cones. The group $\Gamma$ can be lifted to a subgroup $\overline{\Gamma}$ of $SL^{\pm}_{n+1}(\erre)$ preserving each of the convex cones. After this lift, if $\gamma \in \Gamma$, then $\lambda_1$ and $\lambda_{n+1}$ are real and positive.
\end{corollary}

\begin{proposition}   \label{prop:hol in}
Let $\Gamma \subset PGL_{n+1}(\erre)$ be a torsion-free group dividing a strictly convex set $\Omega$. The set $\{x_\gamma \ |\ \gamma \in \Gamma \}$ is dense in $\partial \Omega$, hence $\Omega$ is actually determined by $\Gamma$.
\end{proposition}

\begin{proof}
See \cite{Be1}.
\end{proof}

\begin{proposition}  \label{prop:abs irred}
Let $\Gamma \subset PGL_{n+1}(\erre)$ be a group dividing a strictly convex set $\Omega$. Then $\Gamma$ is absolutely irreducible, i.e. it has no nontrivial invariant subspaces in $\ci\pro^n$.
\end{proposition}

\begin{proof}
If $\Omega \not\simeq \iper^n$, by \cite{Be2}, the Zariski closure of $\Gamma$ is $PGL_{n+1}(\erre)$. If $\Omega \simeq \iper^{n}$, the Zariski closure of $\Gamma$ is $O^+(1,n)$. The linear span of $\overline{\Gamma} \subset M_n(\erre)$ contains the Zariski closure, hence it is $M_n(\erre)$, hence the group $\Gamma$ is absolutely irreducible.
\end{proof}

Let $\Omega \subset \erre\pro^n$ be a properly convex set. We denote by 
$$d:\Omega \times \Omega \freccia \erre_{\geq 0}$$
The Hilbert metric on $\Omega$, a metric that is invariant by projective automorphisms of $\Omega$. For a reference on this metric and on the following facts see \cite[sect. 7]{Ki01}.

Let $M = \Omega / \Gamma$ be a strictly convex projective manifold. The quotient $M$ inherits a Finslerian metric from $\Omega$. If $g$ is a closed geodesic in $M$, its lift $\widetilde{g}$ is a segment of projective line, and it is invariant for an element $\gamma \in \Gamma$. The element $\gamma$ acts on $g$ as a translation of length $\ell_\gamma$. The endpoints of $\widetilde{g}$ in $\partial \Omega$ are precisely the attracting and repelling fixed points of $\gamma$, $x_\gamma$ and $y_\gamma$. The translation length $\ell_\gamma$ can be calculated from the eigenvalues $\lambda_1$ and $\lambda_{n+1}$ as
$$\ell_\gamma = \log_e\left(\frac{\lambda_1}{\lambda_n}\right) $$
The function $\ell:\Gamma \freccia \erre_{>0}$ is called the \nuovo{marked length spectrum} of $M$. The marked length spectrum determines the marked projective structure on $M$, see \cite[thm. 2]{Ki01}.

\subsection{Spaces of convex projective structures}

Let $M$ be a closed $n$-manifold such that the fundamental group $\pi_1(M)$ has trivial virtual center, it is Gromov hyperbolic, and it is torsion free. In particular $M$ is orientable, as a non-orientable manifold has an element of the fundamental group of order $2$.

For example every closed hyperbolic $n$-manifold whose fundamental group is torsion-free satisfies the hypotheses.

As usual we denote by $\mathcal{D}_{\erre\pro^n}(M)$ and $\mathcal{T}_{\erre\pro^n}(M)$ the spaces of based and marked projective structures on $M$.

We denote by $\mathcal{D}_{\erre\pro^n}^c(M) \subset \mathcal{D}_{\erre\pro^n}(M)$ and $\mathcal{T}_{\erre\pro^n}^c(M) \subset \mathcal{T}_{\erre\pro^n}(M)$ the subsets corresponding to convex projective structures on $M$, that are automatically strictly convex as $\pi_1(M)$ is Gromov hyperbolic. 

\begin{proposition}[{[Koskul openness theorem]}]
The subsets $\mathcal{D}_{\erre\pro^n}^c(M)\subset \mathcal{D}_{\erre\pro^n}(M)$ and $\mathcal{T}_{\erre\pro^n}^c(M)\subset \mathcal{T}_{\erre\pro^n}(M)$ are open. 
\end{proposition}

\begin{proof}
See \cite[sec. 6]{Ki01}.
\end{proof}

\begin{proposition}
The holonomy map restricted to $\mathcal{D}_{\erre\pro^n}^c(M)$ and $\mathcal{T}_{\erre\pro^n}^c(M)$ is injective, identifying these spaces with their image, an open subset of $\homom(\pi_1(M),PGL_{n+1}(\erre))$ and $\homom(\pi_1(M),PGL_{n+1}(\erre)) / PGL_{n+1}(\erre)$ respectively.
\end{proposition}

\begin{proof}
By proposition \ref{prop:hol in}, a strictly convex projective manifold $\Omega / \Gamma$ is determined by the group $\Gamma$, hence the holonomy is injective.
\end{proof}

\begin{proposition}
The image of the map 
$$\pi_*:\homom(\pi_1(M),SL^{\pm}_{n+1}(\erre)) \freccia \homom(\pi_1(M),PGL_{n+1}(\erre))$$ 
contains the deformation space $\mathcal{D}_{\erre\pro^n}^c(M)$. This map has a canonical section, identifying $\mathcal{D}_{\erre\pro^n}^c(M)$ with an open subset of $\homom(\pi_1(M),SL^{\pm}_{n+1}(\erre))$. 
\end{proposition}

\begin{proof}
By corollary \ref{corol:lift}, every element of $\mathcal{D}_{\erre\pro^n}^c(M)$ has a canonical lift to $\homom(\pi_1(M),SL^{\pm}_{n+1}(\erre))$.
\end{proof}

\begin{theorem}
The set $\mathcal{D}_{\erre\pro^n}^c(M)$ is canonically identified with a finite union of connected components of $\homom(\pi_1(M),SL_{n+1}(\erre))$, in particular it is a clopen semi-algebraic subset.
\end{theorem}

\begin{proof}
As $M$ is orientable, all the representations corresponding to convex structures on $M$ takes their values in $SL_{n+1}(\erre)$, hence $\mathcal{D}_{\erre\pro^n}^c(M)$ is an open subset of $\homom(\pi_1(M),SL_{n+1}(\erre))$.

Given a group $\Gamma$, let
$$\mathcal{F}_\Gamma = \{\rho \in \homom(\Gamma,SL_{n+1}(\erre)) \ |\ \rho \mbox{ is discrete and faithful and such that } $$
$$\rho(\Gamma) \mbox{ divides a properly convex open set } \Omega \subset \erre\pro^m \}$$

Note that
$$\mathcal{F}_{\pi_1(M)} = \bigcup_{N} \mathcal{D}_{\erre\pro^n}^c(N)$$
where $N$ varies among all closed $n$-manifold with $\pi_1(N) = \pi_1(M)$. Hence $\mathcal{F}_{\pi_1(M)}$ is open in $\homom(\pi_1(M),SL_{n+1}(\erre))$.

A theorem of Benoist (see \cite[thm. 1.1]{Be3}) states that the set $\mathcal{F}_{\pi_1(M)}$ is also closed in $\homom(\pi_1(M),SL_{n+1}(\erre))$,  hence it is a union of connected components of $\homom(\pi_1(M),SL_{n+1}(\erre))$. As $\mathcal{F}_{\pi_1(M)}$ is the disjoint union of the open sets $\mathcal{D}_{\erre\pro^n}^c(N)$, each of them is also closed, in particular $\mathcal{D}_{\erre\pro^n}^c(M)$ is a union of connected components of $\homom(\pi_1(M),SL_{n+1}(\erre))$. 

As $\homom(\pi_1(M),SL_{n+1}(\erre))$ is an affine algebraic set, it has a finite number of connected components, and each of them is a semi-algebraic set.
\end{proof}

\begin{corollary}
The set of closed $n$-manifolds $N$ with $\pi_1(N) = \pi_1(M)$ admitting a convex projective structure is finite.
\end{corollary}

\begin{proof}
It follows from the proof of the theorem.
\end{proof}

Now consider the semi-geometric quotient of $\homom(\pi_1(M),SL_{n+1}(\erre))$, and its image in the character variety 
$$t:\homom(\pi_1(M),SL_{n+1}(\erre)) \freccia \overline{\charat}(\pi_1(M),SL_{n+1}(\erre))$$ 

\begin{proposition}
The image $t(\mathcal{D}_{\erre\pro^n}^c(M))$ can be identified with the space $\mathcal{T}_{\erre\pro^n}^c(M)$, it is a finite union of connected components (and, in particular, a clopen semi-algebraic subset) of $\overline{\charat}(\pi_1(M),SL_{n+1}(\erre))$.
\end{proposition}

\begin{proof}
By proposition \ref{prop:abs irred}, all representations in $\mathcal{D}_{\erre\pro^n}^c(M)$ are absolutely irreducible, hence the image $t(\mathcal{D}_{\erre\pro^n}^c(M))$ can be identified with the space $\mathcal{D}_{\erre\pro^n}^c(M) / PGL_{n+1} = \mathcal{T}_{\erre\pro^n}^c(M)$.

The clopen set $\mathcal{D}_{\erre\pro^n}^c(M)$ is invariant for the action by conjugation of $PGL_{n+1}(\erre)$, hence by corollary \ref{corol:real char} its image $\mathcal{T}_{\erre\pro^n}^c(M)$ is clopen in $\overline{\charat}(\pi_1(M),SL_{n+1}(\erre))$, in particular it is a semi-algebraic set.
\end{proof}

Now we present a result showing that the space $\mathcal{T}_{\erre\pro^n}^c(M)$ is often big enough to be interesting, as there are cases where we know a lower bound on the dimension of this space.

\begin{proposition}
Suppose that $M$ is a closed hyperbolic $n$-manifold containing $r$ two-sided disjoint connected totally geodesic hypersurfaces. Then $\dim \mathcal{T}_{\erre\pro^n}^c(M) \geq r$.
\end{proposition}

\begin{proof}
The space $\mathcal{T}_{\erre\pro^n}(M)$ has a special point $x_0$ corresponding to the hyperbolic structure. In \cite{JM87} it is proven that in this case $\mathcal{T}_{\erre\pro^n}(M)$ contains a ball of dimension $r$ around $x_0$. As $\mathcal{T}_{\erre\pro^n}^c(M)$ is open in $\mathcal{T}_{\erre\pro^n}(M)$ and contains $x_0$, the dimension of $\mathcal{T}_{\erre\pro^n}^c(M)$ is at least $r$.
\end{proof}

\subsection{Compactification}   \label{subsez:compproj}

Let $M$ be a closed $n$-manifold such that the fundamental group $\pi_1(M)$ has trivial virtual center, it is Gromov hyperbolic, and it is torsion free (note that every closed hyperbolic $n$-manifold whose fundamental group is torsion-free satisfies the hypotheses). We want to construct a compactification of the space $\mathcal{T}^{c}_{\erre\pro^n}(M)$ of marked convex projective structures on $M$, using the structure of semi-algebraic set it inherits from its identification with a connected component of $\overline{\charat}(\pi_1(M),SL_n(\erre))$.

For every element $p \in \mathcal{T}^{c}_{\erre\pro^n}(M)$ and $\gamma \in \pi_1(M)$, we denote by $\ell_\gamma(p)$ the translation length of $\gamma$, see the end of subsection \ref{subsez:projective geometry}, and we denote by $e_\gamma(p)$ the ratio $\frac{\lambda_1}{\lambda_{n+1}}$ between the eigenvalues of maximum and minimum modulus of the conjugacy class of matrices $p(\gamma)$. Then the function
$$e_\gamma:\mathcal{T}^{c}_{\erre\pro^n}(M) \freccia \erre_{>0}$$
is a semi-algebraic function on $\mathcal{T}^{c}_{\erre\pro^n}(M)$, such that $\log_e(e_\gamma(p)) = \ell_\gamma(p)$.

Let $\gfamil = {\{e_\gamma\}}_{\gamma \in \pi_1(M)}$. As we said in subsection \ref{subsez:chars}, there exists a finite subset $A \subset \pi_1(M)$ such that the family of functions ${\{ I_\gamma \}}_{\gamma \in A}$ generates the ring of coordinates of $\charat(\pi_1(M),SL_{n+1}(\erre))$.

\begin{lemma}
Let $p \in \mathcal{T}^{c}_{\erre\pro^n}(M)$ and $\gamma \in \pi_1(M)$. Denote by $\lambda_1, \dots, \lambda_{n+1}$ the complex eigenvalues of $p(\gamma)$, with non-increasing absolute values. Then:
$$|\trace(p(\gamma))| \leq (n+1) \frac{\lambda_1}{\lambda_{n+1}} $$
\end{lemma}

\begin{proof}
First we recall that $\lambda_1$ and $\lambda_{n+1}$ are real and positive, hence the statement makes sense. As $|\lambda_1| \geq |\lambda_i|$, then $|\trace(p(\gamma))| \leq (n+1) \lambda_1$. As $p(\gamma)\in SL_{n+1}(\erre)$, then $\lambda_{n+1} \leq 1$.
\end{proof}

\begin{proposition}   \label{prop:proper families}
The family $\famil_A = {\{ e_\gamma \}}_{\gamma \in A}$ is proper.
\end{proposition}

\begin{proof}
Suppose that $(x_n) \subset \mathcal{T}^{c}_{\erre\pro^n}(M)$ is a sequence that is not contained in any compact subset of $\mathcal{T}^{c}_{\erre\pro^n}(M)$. Suppose, by contradiction, that the image $E_{\famil_A}(x_n) \subset {(\erre_{>0})}^{\famil_A}$ is bounded. Then, by previous lemma, the functions ${\{ I_\gamma \}}_{\gamma \in A}$ are bounded on $(x_n)$, hence the sequence $(x_n)$ converges (up to subsequences) to a point of $\charat(\pi_1(M),SL_{n+1}(\erre))$. As $\mathcal{T}^{c}_{\erre\pro^n}(M)$ is closed in $\charat(\pi_1(M),SL_{n+1}(\erre))$, we have a contradiction.
\end{proof}

We have proved that the family $\gfamil$ is a proper family, hence it defines a compactification
$$\overline{\mathcal{T}^{c}_{\erre\pro^n}(M)}_\gfamil = \mathcal{T}^{c}_{\erre\pro^n}(M) \cup \partial_\gfamil \mathcal{T}^{c}_{\erre\pro^n}(M)$$

As the family $\gfamil$ is invariant for the action of the mapping class group of $M$, the action of the mapping class group extends continuously to an action on $\overline{\mathcal{T}^{c}_{\erre\pro^n}(M)}_\gfamil$.

Note that this compactification is constructed taking the limits of the functions $\log_e \circ e_\gamma$, i.e. the limits of the translation length functions $\ell_\gamma$.

In the companion paper \cite{A2} we investigate which objects can be used for the interpretation of the boundary points. Points of the interior part of $\overline{\charat}(\pi_1(M),SL_{n+1}(\erre))$ correspond to representations of $\pi_1(M)$ on $SL_{n+1}(\erre)$, or, geometrically, to actions of the group on a projective space of dimension $n$. Points of the boundary correspond instead to representations of the group in $SL_{n+1}(\cappa)$, where $\cappa$ is a non-archimedean field. In \cite{A2} we find a geometric interpretation of these representations, as actions of the group on tropical projective spaces of dimension $n$.

\section{Compactification of Teichm\"uller spaces}            \label{sez:Teichmuller}

\subsection{Teichm\"uller spaces and character varieties}      \label{subsez:Teichmuller}

Let $\overline{S} = \Sigma_g^k$ be an orientable compact surface of genus $g$, with $k \geq 0$ boundary components and with $\chi(S) < 0$, and let $S$ be its interior. As $S$ is orientable, the image of the holonomy of an hyperbolic structure on $S$ only contains orientation preserving maps, hence it takes values in $PSL_2(\erre)$. We will denote by $\mathcal{T}^{cf}_{\iper^2}(S)$ the deformation space of marked hyperbolic structures on $S$ that are complete and with finite volume. If $S$ is closed ($k=0$) this is just the deformation space of marked hyperbolic structures on $S$. The spaces $\mathcal{T}^{cf}_{\iper^2}(S)$ are usually called \nuovo{Teichm\"uller spaces}.

The holonomy map
$$\hol:\mathcal{T}^{cf}_{\iper^2}(S) \freccia \homom(\pi_1(S), PSL_2(\erre)) / PGL_2(\erre)$$
is injective and it identifies the space $\mathcal{T}^{cf}_{\iper^2}(S)$ with s subset of the quotient $\homom(\pi_1(M), PSL_2(\erre)) / PGL_2(\erre)$.

In \cite[III.1.9]{MS84} it is shown that the Teichm\"uller spaces can be embedded in $\charat(\pi_1(S),SL_2(\erre))$, with an identification that is natural only up to the action of the group $\homom(\pi_1(S), \ze_2)$. The image of this embedding is a connected component of a closed algebraic subset of $\charat(\pi_1(S),SL_2(\erre))$ (see \cite[III.1.9]{MS84}). In particular we can identify $\mathcal{T}^{cf}_{\iper^2}(S)$ with a closed semi-algebraic subset of $\overline{\charat}(\pi_1(S),SL_2(\erre))$.

Consider the family of functions $\gfamil = \{I_\gamma\}_{\gamma \in \pi_1(S)}$. These are functions on $\charat(\pi_1(S),SL_2(\erre))$. The immersion of the Teichm\"uller space $\mathcal{T}^{cf}_{\iper^n}(S)$ in $\charat(\pi_1(S),SL_2(\erre))$ is not canonical, it is well defined only up to the action of the group $\homom(\pi_1(S),\ze_2)$. If two points of $\charat(\pi_1(S),SL_2(\erre))$ have the same orbit for the action of $\homom(\pi_1(S),\ze_2)$, the values of functions $I_\gamma$ on these points coincide up to sign. Hence only $|I_c|$ is a well defined function on $\mathcal{T}^{cf}_{\iper^n}(S)$.

As the functions $I_\gamma$ never vanish on $\mathcal{T}^{cf}_{\iper^2}(S)$, and $\mathcal{T}^{cf}_{\iper^2}(S)$ is connected, they have constant sign. We define the \nuovo{positive trace functions} by choosing a point $x \in \mathcal{T}^{cf}_{\iper^2}(S)$ and then defining for every $\gamma\in\pi_1(S)$ the function $J_\gamma = \mbox{sign}(I_\gamma(x)) I_\gamma$.

On $\mathcal{T}^{cf}_{\iper^2}(S)$ we have $J_\gamma = |I_\gamma|$, hence the functions $J_\gamma$ are canonically well defined on $\mathcal{T}^{cf}_{\iper^2}(S)$, and they are polynomial functions on $\charat(\pi_1(S),SL_2(\erre))$, generating the ring of coordinates.

Positive trace functions are closely related to length functions:
$$\ell_\gamma([h]) = \inf_{\alpha} l_h(\alpha)$$
where $h$ is an hyperbolic metric on $S$, $[h]$ is the corresponding elements of $\mathcal{T}^{cf}_{\iper^2}(S)$, $l_h$ is the function sending a curve in its $h$-length, and the $\inf$ is taken on the set of all closed curves whose free homotopy class is $\gamma$. The relation between trace functions and length function is given by
$$J_\gamma([h]) = 2 \cosh(\frac{1}{2}\ell_c([h]))$$
In particular this implies that $|I_c([h])|\geq 2$ on $\mathcal{T}^{cf}_{\iper^2}(S)$.

\subsection{Compactification}      \label{subsez:compteich}

Let $\overline{S} = \Sigma_g^k$ be an orientable compact surface of genus $g$, with $k \geq 0$ boundary components and with $\chi(\overline{S}) < 0$, and let $S$ be its interior. We want to construct a compactification of the Teichm\"uller space $\mathcal{T}^{cf}_{\iper^2}(S)$, using the structure of semi-algebraic set it inherits from its identification with a closed semi-algebraic subset of $\overline{\charat}(\pi_1(S),SL_2(\erre))$ and $\charat(\pi_1(S),SL_2(\erre))$.

Let $\gfamil = {\{J_\gamma\}}_{\gamma \in \pi_1(S)}$, the positive trace functions, as defined in subsection \ref{subsez:Teichmuller}. As we said in subsection \ref{subsez:chars}, there exists a finite subset $A \subset \pi_1(S)$ such that the family of functions $\famil_A = {\{ J_\gamma \}}_{\gamma \in A}$ generates the ring of coordinates of $\charat(\pi_1(S),SL_2(\erre))$.

\begin{proposition}
Let $A \subset \pi_1(S)$ be a finite subset such that the family of functions ${\{ J_\gamma \}}_{\gamma \in A}$ generates the ring of coordinates of $\charat(\pi_1(S),SL_2(\erre))$. Then the family $\famil_A = {\{ J_\gamma \}}_{\gamma \in A}$ is proper.
\end{proposition}

\begin{proof}
Suppose that $(x_n) \subset \mathcal{T}^{cf}_{\iper^2}(S)$ is a sequence that is not contained in any compact subset of $\mathcal{T}^{cf}_{\iper^2}(S)$. Suppose, by contradiction, that the image $E_{\famil_A}(x_n) \subset {(\erre_{>0})}^{\famil_A}$ is bounded, or, in other words, the functions ${\{ J_\gamma \}}_{\gamma \in A}$ are bounded on $(x_n)$, hence the sequence $(x_n)$ converges (up to subsequences) to a point of $\charat(\pi_1(S),SL_{n+1}(\erre))$. As $\mathcal{T}^{cf}_{\iper^2}(S)$ is closed in $\charat(\pi_1(S),SL_{n+1}(\erre))$, we have a contradiction.
\end{proof}

By the previous proposition, the family $\gfamil$ is proper, hence it defines a compactification
$$\overline{\mathcal{T}^{cf}_{\iper^2}(S)}_\gfamil = \mathcal{T}^{cf}_{\iper^2}(S) \cup \partial_\gfamil \mathcal{T}^{cf}_{\iper^2}(S)$$

As the family $\gfamil$ is invariant for the action of the mapping class group of $S$, the action of the mapping class group extends continuously to an action on $\overline{\mathcal{T}^{cf}_{\iper^2}(S)}_\gfamil$.

We want to prove that this compactification of the Teichm\"uller space $\mathcal{T}^{cf}_{\iper^2}(S)$ is the same as the one constructed by Thurston, see \cite{FLP} for details on Thurston's work. The boundary constructed by Thurston, here denoted by $\overline{\mathcal{T}^{cf}_{\iper^2}(S)}_T = \mathcal{T}^{cf}_{\iper^2}(S) \cup \partial_T \mathcal{T}^{cf}_{\iper^2}(S)$ can be described by using projectivized length spectra of measured laminations, as in \cite{FLP}. If $l$ is a measured lamination, for every element $\gamma \in \pi_1(S)$, the measure of $\gamma$, denoted by $I(l,\gamma)$, is the infimum of the measures of all the closed curves in $S$ that are freely homotopic to $\gamma$. The cone over the boundary, $C(\partial_T \mathcal{T}^{cf}_{\iper^2}(S))$, can be identified with the subset of $\erre^{\pi_1(S)}$ consisting of points of the form ${(I(l,\gamma))}_{\gamma \in \pi_1(S)}$. A sequence $(x_n) \subset \mathcal{T}^{cf}_{\iper^2}(S)$ converges to a point $x_0 \in \partial_T \mathcal{T}^{cf}_{\iper^2}(S)$ if and only if the projectivized length spectra ${[\ell_\gamma(x_n)]}_{\gamma\in \pi_1(S)}$ of the hyperbolic structures converges to the projectivized length spectra of the measured lamination corresponding to $x_0$.

\begin{proposition}      \label{prop:Thurston}
The compactification $\overline{\mathcal{T}^{cf}_{\iper^2}(S)}_\gfamil$ is isomorphic to the compactification $\overline{\mathcal{T}^{cf}_{\iper^2}(S)}_T$ constructed by Thurston. More precisely, the sets $C(\partial_T \mathcal{T}^{cf}_{\iper^2}(S))$ and $C(\partial_T \mathcal{T}^{cf}_{\iper^2}(S))$ coincide when identified with subsets of $\erre^{\pi_1(S)}$, and the notion of convergence of sequences on $\mathcal{T}^{cf}_{\iper^2}(S)$ to the boundary points also coincide.
\end{proposition}

\begin{proof}
If $(x_n) \subset \mathcal{T}^{cf}_{\iper^2}(S)$ is a sequence converging to $l$ in Thurston compactification, then ${I(l,\gamma)}_{\gamma \in \pi_1(S)}$ is the limit of the sequence ${[\ell_\gamma(x_n)]}_{\gamma\in \pi_1(S)}$, and, by the relation $J_\gamma([h]) = 2 \cosh(\frac{1}{2}\ell_c([h]))$ in subsection \ref{subsez:Teichmuller}, this is equal to the limit of the sequence ${[\log_e(J_\gamma(x_n))]}_{\gamma \in \pi_1(S)}$.
\end{proof}

\subsection{Existence of an injective family} \label{subsez:PL struc}

To show that our construction of the boundary defines a piecewise linear structure on it, we only need to find an injective family. First we show a sufficient condition for a family to be proper.

\begin{proposition}
Suppose that $A \subset \pi_1(S)$ has the property that the free homotopy classes of curves in $A$ \nuovo{fill up}, i.e. every free homotopy class of closed curves on the surface has non zero intersection number with at least one of those curves. Then $\famil_A = {\{ J_\gamma\}}_{\gamma \in A}$ is proper.
\end{proposition}

\begin{proof}
We have to show that for every sequence $(x_n) \subset \mathcal{T}^{cf}_{\iper^2}(S)$ that is not contained in a compact subset of $\mathcal{T}^{cf}_{\iper^2}(S)$, the sequence ${(J_\gamma(x_n))}_{\gamma \in A}$ is unbounded in $\erre^{\famil_A}$. Up to subsequences, we can suppose that $(x_n) \freccia x \in \partial_\gfamil \mathcal{T}^{cf}_{\iper^2}(S)$. Let $l$ be a measured foliation associated to $x$ in Thurston's interpretation. As $A$ is a system that fills up, there exists a $\gamma \in A$ such that $I(l,\gamma) \neq 0$, and, this implies that $J_\gamma(x_n)$ is unbounded (see the proof of proposition \ref{prop:Thurston}). Hence $\famil_A$ is proper.
\end{proof}

\begin{proposition}
There exists an injective family $A \subset \pi_1(A)$, consisting on $9g-9+3b$ elements.
\end{proposition}

\begin{proof}
There exists $3g-3+b$ simple curves on $S$ (denoted by $K_1 \dots K_{3g-3+b}$) that decompose it in $2g-2+b$ pair-of-pants, $b$ of them containing a boundary component of $S$. Let $K_i$ be a curve that is the common boundary of two pair-of-pants whose union will be denoted by $P_i$. We denote by $K_i', K_i''$ the two simple closed curves in $P_i$ defined by Thurston in the classification of measured foliation (See \cite{FLP}). Let $A \subset \pi_1(S)$ be the set of the homotopy classes of the curves $K_i$, $K_i'$ and $K_i''$. This set fills up, hence $\famil_A$ is proper, and the map $\partial \pi_\famil:\partial_\gfamil \mathcal{T}^{cf}_{\iper^2}(S) \freccia \partial_{\famil_A} \mathcal{T}^{cf}_{\iper^2}(S)$ is well defined. The fact that this map is injective from Thurston's classification of measured foliations (see \cite{FLP}).
\end{proof}

\subsection{The smallest injective families}

It may be useful to find a set $A \subset \pi_1(S)$ of minimal cardinality such that $\famil_A$ is injective.  Such a set can be found with $6g-5+2b$ elements, just one more than the dimension of $\mathcal{T}^{cf}_{\iper^2}(S)$. What we need is a set of curves that fill up and such that the map ${(I(\cdot,c))}_{c \in A}$ from the set of equivalence classes of measured laminations in $\erre^{\famil_A}$ is injective. If $k=0$ such a set is described in \cite{Ha2}, else it is described in \cite{Ha1}. In the following $A$ will denote a set with these properties.

Let $W = E_{\famil_A}(\mathcal{T}^{cf}_{\iper^2}(S)) \subset \erre^{6g-5+2b}$, a closed semi-algebraic set of dimension less than or equal to the dimension of $\mathcal{T}^{cf}_{\iper^2}(S)$, $6g-6+2b$. The boundary of $\partial_{\famil_A} \mathcal{T}^{cf}_{\iper^2}(S)$ is equal to the boundary $\partial W$. As $\famil_A$ is injective, the boundary $\partial W$ has dimension $6g-7+2b$, and this implies that the dimension of $W$ is exactly $6g-6+2b$. Hence $W$ is an hypersurface in $\erre^{6g-5+2b}$, and the cone over its boundary is contained in a tropical hypersurface. In this way we identify the cone over the boundary of the Teichm\"uller space with a subpolyhedron of a tropical hypersurface in $\erre^{6g-5+2b}$.

If $k>0$ and $A$ is the set described in \cite{Ha1}, the map $E_{\famil_A}$ is also injective. This may be shown using the fact that the map 
$$\mathcal{T}^{cf}_{\iper^2}(S) \ni x \freccia {(\ell_\gamma(x))}_{\gamma \in A} \in \erre^A$$
is injective (see \cite{Ha1}), and that $\famil_A$ is the composition of this map with $\cosh$.

So we have a semi-algebraic homeomorphism from $\mathcal{T}^{cf}_{\iper^2}(S)$ to the closed semi-algebraic hypersurface $W$.

\subsection{Description of the piecewise linear structure}

We have constructed a piecewise linear structure on the boundary of the Teichm\"uller spaces $\mathcal{T}^{cf}_{\iper^2}(S)$. A piecewise linear structure on $\partial_\gfamil \mathcal{T}^{cf}_{\iper^2}(S)$ was already known, it was defined by Thurston. See \cite{Pap} for details on this structure.

\begin{theorem}
The piecewise linear structure defined above on $\partial_\gfamil \mathcal{T}^{cf}_{\iper^2}(S)$ is the same as the one discussed in \cite{Pap}.
\end{theorem}

\begin{proof}
In \cite{Pap} the piecewise linear structure is defined using train tracks. An admissible train track on $S$, is a graph embedded in $S$ satisfying certain conditions. A measure on a train track is a function from the set of the edges in $\erre_{> 0}$, satisfying, again, certain conditions. There exists an enlargement operation associating to every measured admissible train track a measured foliation $f$ on $S$, hence a point of $C(\partial_\gfamil \mathcal{T}^{cf}_{\iper^2}(S))$.

If $\tau$ is a fixed train track with $n$ edges, every measure on $\tau$ may be identified with a point of ${(\erre_{> 0})}^n$, and the subset of all these point is a polyhedral conic subset, that will be denoted by $C_\tau$. The enlargement operation defines a map $\phi_\tau: C_\tau \freccia C(\partial_\gfamil \mathcal{T}^{cf}_{\iper^2}(S))$. If every connected component of $S \setminus \tau$ is a triangle, the image $\phi_\tau(C_\tau)=V_\tau$ is an open subset of $C(\partial_\gfamil \mathcal{T}^{cf}_{\iper^2}(S))$. Moreover the map $\phi_\tau$ is an homeomorphism with its image.

The union of the open sets $V_\tau$ is the whole $C(\partial_\gfamil \mathcal{T}^{cf}_{\iper^2}(S))$. So every $V_\tau$ is identified with $C_\tau$ in such a way that the changes of charts are piecewise linear. This way we have described a piecewise linear atlas for $C(\partial_\gfamil \mathcal{T}^{cf}_{\iper^2}(S))$, the piecewise linear structure defined by Thurston.

We want to show that the identity map is a piecewise linear map if we endow the domain with the Thurston's piecewise linear structure, and the codomain with the structure defined above. This implies that the two structure are equal.

To show this we need to prove that the maps $\phi_\tau$ are piecewise linear if we endow the codomain with the structure defined above. We choose an injective family $A \subset \pi_1(S)$, and the cone $C(\partial_{\famil_A} \mathcal{T}^{cf}_{\iper^2}(S))$ is a subset of $\erre^{\famil_A}$.
The coordinates of the map $\phi_\tau: C_\tau \freccia C(\partial_{\famil_A} \mathcal{T}^{cf}_{\iper^2}(S)) \subset \erre^{\famil_A}$, can be described as follows, for each element $\gamma \in A$ the corresponding coordinate of $\phi_\tau$ is the function that associate to a measure $\mu$ on $\tau$ the number ${(I(f,\gamma))}$, where $f$ is the foliation constructed by the enlargement of the measured train track $(\tau,\mu)$.

For all $\gamma \in A$ it is easy to see that the corresponding coordinate is piecewise linear. We choose a curve $c$ that is freely homotopic to $\gamma$, such that $c$ does not contain any vertex of $\tau$ and such that it intersects every edge transversely. Now we define the function $p_\gamma: C_\tau \freccia \erre$ as the sum of the measures of all the edges intersected by $c$, counted with multiplicity. This function is the restriction of a linear function with positive integer coefficients. The coordinate of $\phi_\tau$ corresponding to $\gamma$ is simply the minimum of all the function $p_\gamma$, and locally the minimum may be taken over a finite numbers of linear functions, so the result is a piecewise linear function.
\end{proof}

\end{document}